\documentclass[twoside,11pt]{article}
\usepackage{jmlr2e}
\usepackage[english]{babel}
\usepackage{amsmath}
\usepackage{amsfonts}
\usepackage{amssymb}

\usepackage{enumerate}
\usepackage{url}

\newtheorem{property}[theorem]{Property}
\newtheorem{algo}[theorem]{Algorithm}
\newtheorem{cor}[theorem]{Corollary}
\newtheorem{rem}[theorem]{Remark}
\newtheorem{ex}[theorem]{Example}
\DeclareMathOperator{\pen}{pen}
\DeclareMathOperator{\id}{id}
\DeclareMathOperator{\tr}{tr}
\DeclareMathOperator{\df}{df}
\DeclareMathOperator{\vect}{vec}
\DeclareMathOperator{\ind}{ind}
\DeclareMathOperator{\var}{Var}

\DeclareMathOperator{\card}{card}
\DeclareMathOperator{\crit}{crit}
\DeclareMathOperator{\diag}{Diag}
\newcommand{\R}{\mathbb{R}}
\newcommand{\E}{\mathbb{E}}
\newcommand{\NN}{\mathbb{N}}
\newcommand{\X}{\mathcal{X}}
\newcommand{\F}{\mathcal{F}}
\newcommand{\G}{\mathcal{G}}
\newcommand{\D}{\mathcal{D}}
\renewcommand{\P}{\mathcal{P}}
\newcommand{\N}{\mathcal{N}}
\newcommand{\M}{\mathcal{M}}
\renewcommand{\a}{\alpha}
\newcommand{\e}{\varepsilon}
\newcommand{\s}{\sigma}

\renewcommand{\r}{\rho}
\renewcommand{\th}{\theta}
\renewcommand{\b}{\beta}
\renewcommand{\S}{\Sigma}
\renewcommand{\l}{\lambda}

\renewcommand{\O}{\Omega}
\renewcommand{\d}{\delta}
\newcommand{\esp}[1]{\mathbb{E} \left[ #1 \right]}
\newcommand{\linspan}[1]{\textrm{span} \left\{ #1 \right\}}
\newcommand{\nor}[1]{|\!|\!| #1 |\!|\!|}
\newcommand{\scal}[2]{\langle #1 , #2 \rangle}

\newcommand{\trsp}{^{\top}}
\newcommand{\norr}[1]{\left\| #1 \right\|_2 }

\renewcommand{\hat}[1]{\widehat{#1}}
\renewcommand{\tilde}[1]{\widetilde{#1}}

\newcommand{\pt}{\enspace .}
\newcommand{\virg}{\enspace ,}
\newcommand{\Eq}[1]{\textrm{Equation}~\eqref{#1}}
\newcommand{\fh}{\widehat{f}}
\newcommand{\Sh}{\widehat{\Sigma}}
\newcommand{\Mh}{\widehat{M}}
\newcommand{\argmin}[1]{\underset{#1}{\operatorname{argmin}}}

\newcommand{\Proba}{\mathbb{P}}
\newcommand{\Ch}{\widehat{C}}
\newcommand{\sigh}{\widehat{\sigma}}
\newcommand{\multitask}{\mathrm{similar}}
\newcommand{\singletask}{\mathrm{ind}}
\newcommand{\simple}{\mathrm{simplified}}
\newcommand{\Mmultiens}{\M_{\multitask}}
\newcommand{\Mmulti}{M_{\multitask}}
\newcommand{\clustering}{\mathrm{clus}}
\newcommand{\segmentation}{\mathrm{interval}}
\newcommand{\Mclustering}{\M_{\clustering}}
\newcommand{\Msegmentation}{\M_{\segmentation}}
\newcommand{\MI}{M_{I}}
\newcommand{\HM}{\textrm{HM}}
\newcommand{\egaldef}{:=}
\newcommand{\paren}[1]{\left( #1 \right)}
\newcommand{\sparen}[1]{(#1)}
\newcommand{\croch}[1]{\left[ #1 \right]}
\newcommand{\scroch}[1]{[#1]}
\newcommand{\set}[1]{\left\{ #1 \right\}}
\newcommand{\sset}[1]{\{ #1 \}}
\newcommand{\absj}[1]{\left\lvert #1 \right\rvert}

\newcommand{\mutilde}{\tilde{\mu}}
\providecommand{\norm}[1]{\left \lVert #1 \right\rVert}
\providecommand{\snorm}[1]{\lVert #1 \rVert}

\title{Multi-task Regression using Minimal Penalties}
\author{\name Matthieu Solnon \email matthieu.solnon@ens.fr  \\
\addr  ENS; Sierra Project-team\\
D\'epartement d'Informatique de l'\'Ecole Normale Sup\'erieure \\
(CNRS/ENS/INRIA UMR 8548)\\
23, avenue d'Italie, CS 81321 \\
75214 Paris Cedex 13, France
\AND
\name Sylvain Arlot \email sylvain.arlot@ens.fr \\
\addr  CNRS; Sierra Project-team\\
D\'epartement d'Informatique de l'\'Ecole Normale Sup\'erieure \\
(CNRS/ENS/INRIA UMR 8548)\\
23, avenue d'Italie, CS 81321 \\
75214 Paris Cedex 13, France
\AND
\name Francis Bach \email francis.bach@ens.fr \\
\addr  INRIA; Sierra Project-team\\
D\'epartement d'Informatique de l'\'Ecole Normale Sup\'erieure \\
(CNRS/ENS/INRIA UMR 8548)\\
23, avenue d'Italie, CS 81321 \\
75214 Paris Cedex 13, France
}
\jmlrheading{13}{2012}{2773-2812}{7/11; Revised 4/12}{9/12}{Solnon, Arlot and Bach}
\ShortHeadings{Multi-task Regression using Minimal Penalties}{Solnon, Arlot and Bach}
\firstpageno{2773}
\editor{Tong Zhang}

\begin{document}
\maketitle
\begin{abstract}%
  In this paper we study the kernel  multiple ridge regression framework, which we refer to as multi-task regression, using penalization techniques. The theoretical analysis of this problem shows that the key element appearing for an optimal calibration is the covariance matrix of the noise between the different tasks. We present a new algorithm to estimate this covariance matrix, based on the concept of minimal penalty, which was previously used in the single-task regression framework to estimate the variance of the noise. We show, in a non-asymptotic setting and under mild assumptions on the target function, that this estimator converges towards the covariance matrix. Then plugging this estimator into the corresponding ideal penalty leads to an oracle inequality. We illustrate the behavior of our algorithm on synthetic examples.
\end{abstract}

\section{Introduction}
A classical paradigm in statistics is that increasing the sample size (that is, the number of observations) improves the performance of the estimators. However, in some cases it may be impossible to increase the sample size, for instance because of experimental limitations. Hopefully, in many situations practicioners can find many related and similar problems, and might use these problems as if more observations were available for the initial problem. The techniques using this heuristic are called ``multi-task'' techniques. In this paper we study the kernel ridge regression procedure in a multi-task framework.

One-dimensional kernel ridge regression, which we refer to as ``single-task'' regression, has been widely studied. As we briefly review in Section~\ref{simple_reg}  one has, given $n$ data points $(X_i,Y_i)_{i=1}^n$, to estimate a function $f$, often the conditional expectation $f(X_i) = \mathbb{E}[Y_i|X_i]$, by minimizing the quadratic risk of the estimator regularized by a certain norm.  A practically important task is to calibrate a regularization parameter, that is, to estimate the regularization parameter directly from data. For kernel ridge regression (a.k.a.~smoothing splines), many methods have been proposed based on different principles, for example, Bayesian criteria through a Gaussian process interpretation \citep[see, e.g.,][]{GP} or generalized cross-validation \citep[see, e.g.,][]{Wah:1990}. In this paper, we focus on the concept of minimal penalty, which was first introduced by \citet{BirgeMassart_min_pen} and \citet{Arl_Mas:2009:pente} for model selection, then extended to linear estimators such as kernel ridge regression by \citet{Arl_Bac:2009:minikernel_long}.

 In this article we consider $p \geq 2$ different (but related) regression tasks, a framework we refer to as ``multi-task'' regression. This setting has already been studied in different papers. Some   empirically show that it can lead to performance improvement  \citep{Thrun96g,Caruana:1997:ML:262868.262872,Bakker:2003:TCG:945365.945370}. \citet{liang10regularization} also obtained a theoretical  criterion (unfortunately non observable) which tells when this phenomenon asymptotically occurs. Several different paths have been followed to deal with this setting. Some consider a setting where $p \gg n$, and formulate a sparsity assumption which enables to use the group Lasso, assuming all the different functions have a small set of common active covariates \citep[see for instance][]{Obozinski_Wainwright_Jordan_2011,Lounici:1277345}. We exclude this setting from our analysis, because of the Hilbertian nature of our problem, and thus will not consider the similarity between the tasks in terms of sparsity, but rather in terms of an Euclidean similarity.   Another theoretical approach has also been taken (see for example, \citet{Brown_Zidek_1980}, \citet{Evgeniou} or \citet{Ando:2005:FLP:1046920.1194905} on semi-supervised learning), the authors often defining a theoretical framework where the multi-task problem can easily be expressed, and where sometimes solutions can be computed. The main remaining theoretical problem  is the calibration of a matricial parameter $M$ (typically of size~$p$), which characterizes the relationship between the tasks and extends the regularization parameter from single-task regression. Because of the high dimensional nature of the problem (i.e., the small number of training observations) usual techniques, like cross-validation, are not likely to succeed. \citet{DBLP:journals/ml/ArgyriouEP08} have a similar approach to ours, but solve this problem by adding a convex constraint to the matrix, which will be discussed at the end of Section \ref{ineg_oracle}.

  Through a penalization technique we show in Section \ref{sec_anal_th} that the only element we have to estimate is the correlation matrix $\S$ of the noise between the tasks. We give here a new algorithm to estimate $\S$, and show that the estimation is sharp enough to derive an oracle inequality for the estimation of the task similarity matrix $M$, both with high probability and in expectation. Finally we give some simulation experiment results and show that our technique correctly deals with the multi-task settings with a low sample-size.

\subsection{Notations} We now introduce some notations, which will be used throughout the article.

\begin{itemize}
\item The integer $n$ is the sample size, the integer $p$ is the number of tasks.
\item For any $n \times p$ matrix $Y$, we define
\[ y = \vect(Y) \egaldef \paren{ Y_{1,1}, \ldots, Y_{n,1}, Y_{1,2}, \ldots, Y_{n,2}, \ldots, Y_{1,p}, \ldots, Y_{n,p} } \in \R^{np}, \]
that is, the vector in which the columns $Y^j \egaldef (Y_{i,j})_{1 \leq i \leq n}$ are stacked.
\item $\mathcal{M}_n(\R)$ is the set of all matrices of size $n$.
\item $\mathcal{S}_p(\R)$ is the set of symmetric matrices of size $p$.
\item $\mathcal{S}_p^+(\R)$ is the set of symmetric positive-semidefinite matrices of size $p$.
\item $\mathcal{S}_p^{++}(\R)$ is the set of symmetric positive-definite matrices of size $p$.
\item $\preceq$ denotes the partial ordering on $\mathcal{S}_p(\R)$ defined by: $A \preceq B$ if and only if $B-A \in \mathcal{S}_p^+(\R)$.
\item $\mathbf{1}$ is the vector of size $p$ whose components are all equal to $1$.
\item $\norr{\cdot}$ is the usual Euclidean norm on $\R^{k}$ for any $k \in \NN$: $\forall u \in \R^k$, $\norr{u}^2 \egaldef \sum_{i=1}^k u_i^2$.
\end{itemize}

\section{Multi-task Regression: Problem Set-up \label{sec_anal_th}}

We consider $p$ kernel ridge regression tasks. Treating them simultaneously and sharing their common structure (e.g., being close in some metric space) will help in reducing the overall prediction error.

\subsection{Multi-task with a Fixed Kernel}

Let $\X$ be some set and $\F$ a set of real-valued functions over $\X$. We suppose $\F$ has a reproducing kernel Hilbert space (RKHS) structure \citep{Aronszajn}, with kernel $k$ and feature map $\Phi: \X \to \F$.  We observe $\D_n = (X_i,Y_i^1,\ldots,Y_i^p)_{i=1}^n \in (\X \times  \R^p)^{n}$, which gives us the positive semidefinite kernel matrix $K = (k(X_i,X_{\ell}))_{1 \leq i , \ell \leq n} \in \mathcal{S}_n^+(\R)$.
For each task $j \in \{ 1,\ldots,p\}$, $\D_n^j = (X_i,y_i^j)_{i=1}^n$ is a sample with distribution $\P_j$, for which a simple regression problem has to be solved.
In this paper we consider for simplicity that the different tasks have the same design $(X_i)_{i=1}^n$. When the designs of the different tasks are different the analysis is carried out similarly by defining $X_i = (X_i^1,\dots,X_i^p)$, but the notations would be more complicated.

We now define the model. We assume $(f^1,\ldots,f^p) \in \F^p$, $\S$ is a symmetric positive-definite matrix of size $p$ such that the vectors $(\e_i^j)_{j=1}^p$ are i.i.d.~with normal distribution $\N(0,\S)$, with mean zero and covariance matrix $\Sigma$, and
\begin{equation} \label{modele}
  \forall i \in \{ 1,\ldots,n\},\forall j \in \{ 1,\ldots,p\},  ~  y_i^j = f^j(X_i) + \e_i^j \pt
\end{equation}
This means that, while the observations are independent, the outputs of the different tasks can be correlated, with correlation matrix $\S$ between the tasks. We now place ourselves in the fixed-design setting, that is, $(X_i)_{i=1}^n$ is deterministic and the goal is to estimate $\left( f^1(X_i),\dots,f^p(X_i) \right)_{i=1}^n $.
Let us introduce some notation:
\begin{itemize}
\item $\mu_{\min} = \mu_{\min}(\S)$ (resp.~$\mu_{\max}$) denotes the smallest (resp.~largest) eigenvalue of $\S$.
\item $c(\S) \egaldef \mu_{\max} / \mu_{\min}$ is the condition number of $\S$.
\end{itemize}
To obtain compact equations, we will use the following definition:
\begin{definition}
  We denote by $F$ the $n \times p$ matrix $(f^j(X_i))_{1 \leq i \leq n \, , \, 1 \leq j \leq p}$ and introduce the vector  $f \egaldef \vect(F) = ( f^1(X_1),\dots,f^1(X_n),\dots, f^p(X_1),\dots,f^p(X_n) )  \in \R^{np}$, obtained by stacking the columns of $F$.
  Similarly we define $Y \egaldef (y_i^j)\in \M_{n\times p}(\R)$, $y \egaldef \vect(Y)$, $E\egaldef(\e_i^j)\in \M_{n\times p}(\R)$ and $\e \egaldef \vect(E)$.
\end{definition}

In order to estimate $f$, we use a regularization procedure, which extends the classical ridge regression of the single-task setting. Let $M$ be a  $p\times p$ matrix, symmetric and positive-definite. Generalizing the work of \citet{Evgeniou}, we estimate $(f^1,\dots,f^p)\in\F^p$ by
\begin{equation}\label{min_multi_M}
  \fh_M \in \argmin{g \in \F^p} \set{ \frac{1}{np} \sum_{i=1}^n \sum_{j=1}^p (y_i^j - g^j(X_i))^2 + \sum_{j=1}^p \sum_{\ell=1}^p M_{j,l}\scal{g^j}{g^{\ell}}_{\F} } \pt
\end{equation}
 Although $M$ could have a general unconstrained form we may restrict $M$ to certain forms, for either computational or statistical reasons.

\begin{rem}
  Requiring that $M \succeq 0$ implies that \Eq{min_multi_M} is a convex optimization problem, which can be solved through the resolution of a linear system, as explained later. Moreover it allows an RKHS interpretation, which will also be explained later.
\end{rem}

\begin{ex}\label{example_ind}
  The case where the $p$ tasks are treated independently can be considered in this setting: taking $M = M_{\ind}(\l) \egaldef \frac{1}{p}\diag(\l_1,\dots,\l_p)$ for any $\l \in \R^p$ leads to the criterion
  \begin{equation} \label{eq.example_ind.crit}
    \frac{1}{p} \sum_{j=1}^p \croch{  \frac{1}{n} \sum_{i=1}^n (y_i^j - g^j(X_i))^2 + \l_j \|g^j\|_{\F}^2 } \virg
  \end{equation}
  that is, the sum of the single-task criteria described in Section~\ref{simple_reg}. Hence, minimizing \Eq{eq.example_ind.crit} over $\l \in \R^p$ amounts to solve \emph{independently} $p$ single task problems.
\end{ex}
\begin{ex} \label{example_Mlm}
As done by \citet{Evgeniou}, for every $\l,\mu \in \left(0,+\infty\right)^2$, define
$$
\hspace{-1cm} \Mmulti(\l,\mu) \egaldef (\l+p\mu) I_p -  \mu \boldsymbol{1}\boldsymbol{1}\trsp = \begin{pmatrix}\l+(p-1)\mu & & -\mu \\ & \ddots &  \\ -\mu & & \l+(p-1)\mu \end{pmatrix} \pt
$$
Taking $M=\Mmulti(\l,\mu)$ in \Eq{min_multi_M} leads to the criterion
\begin{equation} \label{eq.example_Mlm.crit}
\frac{1}{np} \sum_{i=1}^n \sum_{j=1}^p (y_i^j - g^j(X_i))^2 + \l \sum_{j=1}^p \norm{g^j}_{\F}^2 + \frac{\mu}{2} \sum_{j=1}^p \sum_{k=1}^p \norm{g^j-g^k}_{\F}^2 \pt
\end{equation}
Minimizing \Eq{eq.example_Mlm.crit} enforces a regularization on both the norms of the functions $g^j$ and the norms of the differences $g^j-g^k$. Thus, matrices of the form $\Mmulti(\l,\mu)$ are useful when the functions $g^j$ are assumed to be similar in $\F$. One of the main contributions of the paper is to go beyond this case and learn from data a more general similarity matrix $M$ between tasks.
\end{ex}
\begin{ex} \label{example_clustering}
We extend Example \ref{example_Mlm} to the case where the $p$ tasks consist of two groups of close tasks. Let $I$ be a subset of $\{1,\dots,p\}$, of cardinality $1\leq k \leq p-1$. Let us denote by $I^c$ the complementary of $I$ in  $\{1,\dots,p\}$,  $\boldsymbol{1}_{I}$ the vector $v$ with components $v_i = \boldsymbol{1}_{i \in I}$, and $\diag(I)$ the diagonal matrix $d$ with components $d_{i,i} = \boldsymbol{1}_{i \in I}$. We then define
$$
  \MI(\l,\mu,\nu) \egaldef \l I_p+\mu \diag(I)+\nu\diag(I^c)- \frac{\mu}{k} \boldsymbol{1}_{I}\boldsymbol{1}_{I}\trsp -\frac{\nu}{p-k} \boldsymbol{1}_{I^c}\boldsymbol{1}_{I^c}\trsp \pt
$$
This matrix leads to the following criterion, which enforces a regularization on both the norms of the functions $g^j$ and the norms of the differences $g^j-g^k$ inside the groups $I$ and~$I^c$:
\begin{equation} \label{eq.example_clus.crit}
\frac{1}{np} \sum_{i=1}^n \sum_{j=1}^p (y_i^j - g^j(X_i))^2 + \l \sum_{j=1}^p \norm{g^j}_{\F}^2 + \frac{\mu}{2k} \sum_{j\in I} \sum_{k\in I} \norm{g^j-g^k}_{\F}^2+ \frac{\nu}{2(p-k)} \sum_{j\in I^c} \sum_{k\in I^c} \norm{g^j-g^k}_{\F}^2 \pt
\end{equation}
As shown in Section \ref{simulations}, we can estimate the set $I$ from data (see \citealp{jacob:clustered} for a more general formulation).
\end{ex}
\begin{rem}\label{rem_easytooptimize}
Since $I_p$ and $\boldsymbol{1}\boldsymbol{1}\trsp$ can be diagonalized simultaneously, minimizing \Eq{eq.example_Mlm.crit} and \Eq{eq.example_clus.crit} is quite easy: it only demands optimization over two independent parameters, which can be done with the procedure of \citet{Arl_Bac:2009:minikernel_long}.
\end{rem}

\begin{rem}
As stated below (Proposition~\ref{prodscalG}), $M$ acts as a scalar product between the tasks. Selecting a general matrix $M$ is thus a way to express a similarity between tasks.
\end{rem}
Following \citet{Evgeniou}, we define the vector-space $\G$ of real-valued functions over $\X \times \{1,\dots,p\}$ by
\begin{equation*}
  \G \egaldef \set{ g: \X \times \{1,\dots,p\} \to \R \, / \,
  \forall j \in \set{1,\dots,p} \, , \, g(\cdot,j) \in \F } \pt
\end{equation*}
We now define a bilinear symmetric form over $\G$,
\begin{equation*}
 \forall g, h \in \G \virg \quad  \scal{g}{h}_{\G} \egaldef \sum_{j=1}^p \sum_{l=1}^p M_{j,l}\scal{g(\cdot,j)}{h(\cdot,l)}_{\F} ,
 \end{equation*}
which is a scalar product as soon as $M$ is positive semi-definite (see proof in Appendix~\ref{dem_prop_1}) and leads to a RKHS
(see proof in Appendix~\ref{dem_cor_RKHS}):
\begin{proposition}\label{prodscalG}
With the preceding notations $\scal{\cdot}{\cdot}_{\G}$ is a scalar product on $\G$.
\end{proposition}
\begin{cor}\label{cor_GRKHS} $ (\G,\scal{\cdot}{\cdot}_{\G})$ is a RKHS.
\end{cor}
In order to write down the kernel matrix in compact form, we introduce the following notations.
\begin{definition}[Kronecker Product]
  Let $A \in \mathcal{M} _{m,n}(\R)$, $B \in \mathcal{M}_{p,q}(\R)$. We define the Kronecker product $A \otimes B$ as being the $(mp) \times (nq)$ matrix built with $p \times q$ blocks, the block of index $(i,j)$ being $A_{i,j}\cdot B$:
\begin{equation*}
  A \otimes B =
\begin{pmatrix}
  A_{1,1} B & \ldots & A_{1,n} B  \\
  \vdots & \ddots & \vdots \\
  A_{m,1} B & \ldots & A_{m,n} B  \\
\end{pmatrix} \pt
\end{equation*}
\end{definition}
The Kronecker product is a widely used tool to deal with matrices and tensor products. Some of its classical  properties are given in Section~\ref{proof_propmulti}; see also \citet{Horn_Johnson_Matrix_analysis}.

\begin{proposition} \label{prop_noy_multi}
The kernel matrix associated with the design $\widetilde{X} \egaldef (X_i,j)_{i,j} \in \X \times \sset{1,\dots,p}$ and the RKHS $ (\G,\scal{\cdot}{\cdot}_{\G})$ is  $\tilde{K}_M \egaldef M^{-1} \otimes K$.
\end{proposition}

Proposition \ref{prop_noy_multi} is proved in Appendix \ref{dem_prop_noy_multi}. We can then apply the representer's theorem \citep{Schölkopf_Smola_learning} to the minimization problem \eqref{min_multi_M} and deduce that $\fh_{M} = A_{M} y$ with
\begin{equation*}
A_{M}  = A_{M,K}
\egaldef \tilde{K}_{M}(\tilde{K}_{M}+np I_{np})^{-1} = (M^{-1} \otimes K)\left((M^{-1} \otimes K)+np I_{np}\right)^{-1} \pt
\end{equation*}

\subsection{Optimal Choice of the Kernel \label{kernel_choice}}

Now when working in multi-task regression, a set $\M \subset \mathcal{S}_p^{++}(\R)$ of matrices $M$ is given, and the goal is to select the ``best'' one, that is, minimizing over $M$ the quadratic risk $n^{-1}\|\fh_M-f\|_{2}^2$.
For instance, the single-task framework corresponds to $p=1$ and $\M = (0,+\infty)$.
The multi-task case is far richer. The oracle risk is defined as
\begin{equation}\label{oracle_risk}
 \inf_{M \in \M}\set{ \norr{\fh_{M}-f}^2 } \pt
\end{equation}
The ideal choice, called the oracle, is any matrix
\begin{equation*}
M^{\star} \in \argmin{M \in \M}\set{ \norr{\fh_{M}-f}^2 } \pt
\end{equation*}
Nothing here ensures the oracle exists. However in some special cases (see for instance Example \ref{ex_calcul_oracle}) the infimum of $ \|\fh_{M}-f\|^2$ over the set $\{\fh_M,~M \in \M\}$ may be attained by a function $f^* \in \F^p$---which we will call ``oracle'' by a slight abuse of notation---while the former problem does not have a solution.

From now on we always suppose that the infimum of $\{ \|\fh_{M}-f\|^2 \}$ over $\M$ is attained by some function $f^{\star} \in \F^p$. However the oracle $M^{\star}$ is not an estimator, since it depends on $f$.

\begin{ex}[Partial computation of the oracle in a simple setting]\label{ex_calcul_oracle}
It is possible in certain simple settings to exactly compute the oracle (or, at least, some part of it). Consider for instance the set-up where the $p$ functions are taken to be equal (that is, $f^1=\dots =f^p$). In this setting it is natural to use the set
\begin{equation*}
\Mmultiens \egaldef \set{ \Mmulti \paren{\l,\mu} = (\l+p\mu) I_p - \frac{\mu}{p} \boldsymbol{1}\boldsymbol{1}\trsp \, / \, (\l,\mu) \in (0,+\infty)^2 }
\pt
\end{equation*}
Using the estimator $\fh_{M} = A_{M} y$ we can then compute the quadratic risk using the bias-variance decomposition given in Equation~\eqref{eqfond1}\textup{:}
\begin{equation*}
  \esp{\norr{\fh_{M} -f}^2} = \|(A_{M}-I_{np})f\|_2^2 + \tr(A_{M}\trsp A_{M}\cdot(\S \otimes I_n))  \pt
\end{equation*}
Computations \textup{(}reported in Appendix \ref{app_calcul_oracle}\textup{)} show that, with the change of variables $\tilde{\mu}=\l+p\mu$, the bias does not depend on $\tilde{\mu}$ and the variance is a decreasing function of $\tilde{\mu}$.
Thus the oracle is obtained when $\tilde{\mu}=+\infty$, leading to a situation where the oracle functions $f^{1,\star},\dots,f^{p,\star}$ verify $f^{1,\star}=\dots=f^{p,\star}$.
It is also noticeable that, if one assumes the maximal eigenvalue of $\S$ stays bounded with respect to $p$, the variance is of order $\mathcal{O}(p^{-1})$ while the bias is bounded  with respect to $p$.
\end{ex}
As explained by \citet{Arl_Bac:2009:minikernel_long}, we choose
\begin{equation*}
  \Mh \in \argmin{M \in \M}\{\crit(M)\}  \quad \mbox{with} \quad \crit(M) = \frac{1}{np} \norm{y - \fh_{M} }_2^2 + \pen(M) \virg
\end{equation*}
where the penalty term $\pen(M)$ has to be chosen appropriately.

\begin{rem}
  Our model \eqref{modele} does not constrain the functions $f^1,\dots,f^p$. Our way to express the similarities between the tasks (that is, between the $f^j$) is via the set $\M$, which represents the a priori knowledge the statistician has about the problem. Our goal is to build an estimator whose risk is the closest possible to the  oracle risk. 
  Of course using an inappropriate set $\M$ (with respect to the target functions $f^1,\dots,f^p$) may lead to bad overall performances. Explicit multi-task settings are given in Examples \ref{example_ind}, \ref{example_Mlm} and \ref{example_clustering} and through simulations in Section \ref{simulations}.
\end{rem}

The unbiased risk estimation principle \citep[introduced by][]{Aka70} requires
\begin{equation*}
  \esp{\crit(M)} \approx \esp{\frac{1}{np} \norm{ \fh_{M} - f }_2^2} \virg
\end{equation*}
which leads to the (deterministic) \emph{ideal penalty}
\begin{equation*}
  \pen_{\id}(M) \egaldef \esp{ \frac{1}{np} \| \fh_{M} - f \|_2^2} - \esp{\frac{1}{np} \norm{ y-\fh_{M} }_2^2} \pt
\end{equation*}
Since $\fh_{M} = A_{M}y$ and $y = f + \e$, we can write
\begin{align*}
  \norm{ \fh_{M}-y }_2^2 & = \norm{ \fh_{M} -f }_2^2 + \norm{ \e }_2^2 -2\scal{\e}{A_{M}\e} + 2\scal{\e}{(I_{np}-A_{M})f} \pt
\end{align*}
Since $\e$ is centered and $M$ is deterministic,  we get, up to an additive factor independent of~$M$,
\begin{equation*}
  \pen_{\id}(M) = \frac{2\esp{\scal{\e}{A_{M}\e}}}{np}\virg
\end{equation*}
that is, as the covariance matrix of $\e$ is $\S \otimes I_{n}$,
\begin{equation} \label{defpenidmulti}
  \pen_{\id}(M) = \frac{2\tr \big( A_{M}\cdot(\S \otimes I_n) \big) }{np}\pt
\end{equation}
In order to approach this penalty as precisely as possible, we have to sharply estimate $\S$.
In the single-task case, such a problem reduces to estimating the variance $\s^2$ of the noise and was tackled by \citet{Arl_Bac:2009:minikernel_long}.
Since our approach for estimating $\S$ heavily relies on these results, they are summarized in the next section.

 Note that estimating $\Sigma$ is a mean towards estimating $M$.
 The technique we develop later for this purpose is not purely a multi-task technique, and may also be used in a different context.

\section{Single Task Framework: Estimating a Single Variance}
\label{simple_reg}
This section recalls some of the main results from \citet{Arl_Bac:2009:minikernel_long} which can be considered as solving a special case of Section~\ref{sec_anal_th}, with $p=1$, $\S=\s^2>0$ and $\M = [0,+\infty]$. Writing $M = \l$ with $\l \in [0,+\infty]$, the regularization matrix is
$$
\forall \l \in (0,+\infty) \, , \quad A_{\l} = A_{\l,K} = K(K+n\l I_n)^{-1}  \virg
$$
$A_0=I_n$ and $A_{+\infty}=0$; the ideal penalty becomes
\begin{equation*}
  \pen_{\id}(\l) = \frac{2\s^2\tr(A_{\l})}{n} \pt
\end{equation*}
By analogy with the case where $A_{\l}$ is an orthogonal projection matrix, $\df(\l) \egaldef \tr(A_{\l})$ is called the effective degree of freedom, first introduced by \citet{mallows1973}; see also the work by \citet{Zhang}.
The ideal penalty however depends on $\s^2$; in order to have a fully data-driven penalty we have to replace $\s^2$ by an estimator $\sigh^2$ inside $\pen_{\id}(\l)$.
For every $\l\in[0,+\infty]$, define
\begin{equation*} \label{defpenmin}
  \pen_{\min}(\l) =   \pen_{\min}(\l,K) \egaldef \frac{(2\tr(A_{\l,K})-\tr(A_{\l,K}\trsp A_{\l,K}))}{n} \pt
\end{equation*}
We shall see now that it is a \emph{minimal penalty} in the following sense. If for every $C>0$
\begin{equation*}
  \hat{\l}_0(C) \in \argmin{\l \in [0,+\infty]} \set{\frac{1}{n}\norm{ A_{\l,K} Y -Y }_2^2 + C\pen_{\min}(\l,K) } \virg
\end{equation*}
then---up to concentration inequalities---$\hat{\l}_0(C)$ acts as a mimimizer of
\begin{equation*}
  g_C(\l) = \esp{\frac{1}{n}\norm{ A_{\l} Y -Y }_2^2 + C\pen_{\min}(\l)} - \s^2= \frac{1}{n}\norm{ (A_{\l}-I_n)f }_2^2 + (C-\s^2)\pen_{\min}(\l) \pt
\end{equation*}
The former theoretical arguments show that
\begin{itemize}
  \item if $C<\s^2$, $g_C(\l)$ decreases with $\df(\l)$ so that $\df(\hat{\l}_0(C))$ is huge: the procedure overfits;
  \item if $C>\s^2$, $g_C(\l)$ increases with $\df(\l)$ when $\df(\l)$ is large enough  so that $\df(\hat{\l}_0(C))$ is much smaller than when $C<\s^2$.
\end{itemize}
The following algorithm was introduced by \citet{Arl_Bac:2009:minikernel_long} and uses this fact to estimate $\s^2$.
\begin{algo} \label{algo1}
\begin{enumerate}
  \item[] {\bf Input:} $Y \in \R^n$, $K \in \mathcal{S}_n^{++}(\R)$
  \item For every $C >0$, compute \begin{equation*} \widehat{\l}_0(C) \in \argmin{\l \in [0,+\infty]} \set{\frac{1}{n}\norm{ A_{\l,K} Y -Y }_2^2 + C\pen_{\min}(\l,K) } \pt \end{equation*}
  \item {\bf Output:} $\Ch$ such that $\df(\widehat{\l}_0(\Ch)) \in [n/10,n/3]$.
\end{enumerate}
\end{algo}
An efficient algorithm for the first step of Algorithm~\ref{algo1} is detailed by \citet{Arl_Mas:2009:pente}, and we discuss the way we implemented Algorithm~\ref{algo1} in Section \ref{simulations}.
The output $\Ch$ of Algorithm~\ref{algo1} is a provably consistent estimator of $\s^2$, as stated in the following theorem.
\begin{theorem}[Corollary of Theorem~1 of \citealp{Arl_Bac:2009:minikernel_long}] \label{thmono}
  Let $\b = 150$.
  Suppose \linebreak[4] $\e \sim \mathcal{N}(0,\s^2 I_n)$ with $\s^2 >0$, and
  that $\l_0 \in (0,+\infty )$ and $d_n \geq 1$ exist such that
  \begin{equation} \label{eq.hypThm1}
    \df(\l_0) \leq \sqrt{n} \mbox{ and } \frac{1}{n} \norm{ (A_{\l_0}-I_n)F }_2^2 \leq d_n \s^2 \sqrt{\frac{\ln n}{n}} \pt
  \end{equation}
  Then for every $\d \geq 2$, some constant $n_0(\d)$ and an event $\Omega$ exist such that $\Proba(\Omega) \geq 1-n^{-\d}$ and if $n \geq n_0(\d)$, on $\Omega$,
  \begin{equation} \label{ineqmonosigma}
    \left( 1-\b(2+\d) \sqrt{\frac{\ln n}{n}}\right)\s^2 \leq \Ch \leq \left( 1+\b(2+\d) d_n \sqrt{\frac{\ln(n)}{n}} \right)\s^2 \pt
  \end{equation}
\end{theorem}
\begin{rem}
  The values $n/10$ and $n/3$ in Algorithm~\ref{algo1} have no particular meaning and can be replaced by $n/k$, $n/k'$, with $k > k' > 2$. Only $\beta$ depends on $k$ and $k'$. Also the bounds required in Assumption \eqref{eq.hypThm1} only impact the right hand side of Equation \eqref{ineqmonosigma} and are chosen to match the left hand side. See Proposition 10 of \citet{Arl_Bac:2009:minikernel_long} for more details.
\end{rem}

\section{Estimation of the Noise Covariance Matrix $\S$ \label{estimation_Sigma}}
Thanks to the results developped by \citet{Arl_Bac:2009:minikernel_long} (recapitulated in Section~\ref{simple_reg}), we know how to estimate a variance for any one-dimensional problem. In order to estimate $\S$, which has $p(p+1)/2$ parameters, we can use several one-dimensional problems.  Projecting $Y$ onto some direction $z \in \R^p$ yields
\begin{equation}\label{Pz}
  Y_z \egaldef Y\cdot z = F\cdot z + E\cdot z = F_z + \e_z \virg
\end{equation}
with $\e_z \sim \N(0,\s^2_z I_n)$ and $\s_z^2 \egaldef \var\scroch{ \e \cdot z } = z\trsp \S z$. Therefore, we will estimate $\s_z^2$ for $z \in \mathcal{Z}$ a well chosen set, and use these estimators to build back an estimation of $\S$.

We now explain how to estimate $\S$ using those one-dimensional projections.

\begin{definition}
  Let $a(z)$ be the output $\Ch$ of Algorithm~\ref{algo1} applied to problem \eqref{Pz}, that is, with inputs $Y_z \in \R^n$ and $K \in \mathcal{S}_n^{++}(\R)$.
\end{definition}

The idea is to apply Algorithm~\ref{algo1} to the elements $z$ of a carefully chosen set $\mathcal{Z}$. Noting $e_i$ the $i$-th vector of the canonical basis of $\R^p$, we introduce $\mathcal{Z} = \{e_i,~ i \in \{1,\dots,p\}\} \cup \{e_i+e_j,~ 1 \leq i < j \leq p \}$. We can see that $a(e_i)$ estimates $\S_{i,i}$, while $a(e_i + e_j)$ estimates $\S_{i,i} + \S_{j,j} + 2\S_{i,j}$. Henceforth, $\S_{i,j}$ can be estimated by $(a(e_i+e_j)-a(e_i)-a(e_j))/2$. This leads to the definition of the following map $J$, which builds a symmetric matrix using the latter construction.
\begin{definition}\label{defJ}
  Let $J: \R^{\frac{p(p+1)}{2}} \to \mathcal{S}_p(\R)$ be defined by
  \begin{align*}
    J(a_1,\dots,a_p,a_{1,2},\dots,a_{1,p},\dots,a_{p-1,p})_{i,i} &= a_i ~\textrm{if}~ 1 \leq i \leq p \virg \\
    J(a_1,\dots,a_p,a_{1,2},\dots,a_{1,p},\dots,a_{p-1,p})_{i,j} &= \frac{a_{i,j}-a_i-a_j}{2} \mbox{ if } 1 \leq i < j  \leq p \pt
  \end{align*}
\end{definition}
This map is bijective, and for all $B \in \mathcal{S}_p(\R)$
\begin{equation*}
  J^{-1}(B) = \left(B_{1,1},\dots,B_{p,p},B_{1,1}+B_{2,2}+2B_{1,2},\dots,B_{p-1,p-1}+B_{p,p}+2B_{p-1,p} \right) \pt
\end{equation*}
This leads us to defining the following estimator of $\S$:
\begin{equation} \label{def_Sigma}
  \widehat{\S} \egaldef J\left(a(e_1),\dots,a(e_p),a(e_1+e_2),\dots,a(e_1+e_p),\dots,a(e_{p-1}+e_p)\right) \pt
\end{equation}

\begin{rem} \label{rem_algo_simplifie}
  If a diagonalization basis $(e_1',\dots,e_p')$ \textup{(}whose basis matrix is $P$\textup{)} of $\S$ is known, or if $\S$ is diagonal, then a simplified version of the  algorithm defined by \Eq{def_Sigma} is
\begin{equation}\label{algo_simplifie}
  \hat{\S}_{\simple} = P\trsp \diag(a(e_1'),\dots,a(e_p')) P\pt
\end{equation}
This algorithm has a smaller computational cost and leads to better theoretical bounds \textup{(}see Remark \ref{rem_algo_simple} and Section~\ref{result_diagonalizable}\textup{)}.
\end{rem}

Let us recall that $\forall \l \in (0,+\infty)$, $A_{\l}=A_{\l,K}=K(K+n\l I_n)^{-1}$.
Following \citet{Arl_Bac:2009:minikernel_long} we make the following assumption from now on:
\begin{equation}\label{Hdf}
  \left.
    \begin{aligned}
      & \forall j \in \set{1, \ldots, p} , \, \exists \l_{0,j} \in (0,+\infty) \, ,  \\
      & \qquad \df(\l_{0,j}) \leq \sqrt{n} \quad \mbox{and} \quad \frac{1}{n} \norm{ (A_{\l_{0,j}}-I_n) F_{e_j} }_2^2 \leq \S_{j,j} \sqrt{\frac{\ln n}{n}}
    \end{aligned}
    \enspace \right\}
\end{equation}
We can now state the first main result of the paper.
\begin{theorem} \label{propmulti}
  Let  $\Sh$ be defined by \Eq{def_Sigma}, $\a = 2$ and assume \eqref{Hdf} holds.
  For every $\d \geq 2$, a constant $n_0(\d)$, an absolute constant $L_1>0$ and an event $\Omega$ exist such that $\Proba(\Omega) \geq 1-  p(p+1)/2\times n^{-\d}$ and if $n \geq n_0(\d)$, on $\Omega$,
  \begin{align} \label{Schap_conv}
    & \hspace{-0.95cm}  (1-\eta)\S \preceq \widehat{\S} \preceq (1+\eta) \S  \\
    \mbox{where} \qquad  \eta &\egaldef L_1(2+\d) p \sqrt{\frac{\ln(n)}{n}} c(\S)^2 \notag
    \pt
  \end{align}
\end{theorem}
Theorem~\ref{propmulti} is proved in Section~\ref{proof_propmulti}.
It shows $\widehat{\S}$ estimates $\S$ with a ``multiplicative'' error controlled with large probability, in a non-asymptotic setting.
The multiplicative nature of the error is crucial for deriving the oracle inequality stated in Section~\ref{ineg_oracle}, since it allows to show the ideal penalty defined in Equation~\eqref{defpenidmulti} is precisely estimated when $\S$ is replaced by $\Sh$.

An important feature of Theorem~\ref{propmulti} is that it holds under very mild assumptions on the mean $f$ of the data (see Remark~\ref{rq.Hdf}).
Therefore, it shows $\Sh$ is able to estimate a covariance matrix {\em without prior knowledge on the regression function}, which,  to the best of our knowledge, has never been obtained in multi-task regression.

\begin{rem}[Scaling of $(n,p)$ for consistency]
  A sufficient condition for ensuring $\widehat{\S}$ is a consistent estimator of $\S$ is
  \begin{equation*}
    p c(\S)^2\sqrt{ \frac{\ln(n)}{n}} \longrightarrow 0 \virg
  \end{equation*}
  which enforces a scaling between $n$, $p$ and $c(\Sigma)$.
  Nevertheless, this condition is probably not necessary since the simulation experiments of Section~\ref{simulations} show that $\Sigma$ can be well estimated \textup{(}at least for estimator selection purposes\textup{)} in a setting where $\eta \gg 1$.
\end{rem}

\begin{rem}[On assumption \eqref{Hdf}] \label{rq.Hdf}
  Assumption \eqref{Hdf} is a single-task assumption \textup{(}made independently for each task\textup{)}.
  The upper bound $\sqrt{\ln(n)/n}$ can be multiplied by any factor $1 \leq d_n \ll \sqrt{n/\ln(n)}$ \textup{(}as in Theorem~\ref{thmono}\textup{)}, at the price of multiplying $\eta$ by $d_n$ in the upper bound of \Eq{Schap_conv}. More generally the bounds on the degree of freedom and the bias in \eqref{Hdf} only influence the upper bound of \Eq{Schap_conv}. The rates are chosen here to match the lower bound, see Proposition 10 of \citet{Arl_Bac:2009:minikernel_long} for more details.

  Assumption \eqref{Hdf} is rather classical in model selection, see \citet{Arl_Bac:2009:minikernel_long} for instance.
  In particular, \textup{(}a weakened version of\textup{)} \eqref{Hdf} holds if the bias $n^{-1} \|(A_{\l}-I_n)F_{e_i}\|_2^2$ is bounded by $C_1 \tr(A_{\l})^{-C_2}$, for some $C_1,C_2>0$.
\end{rem}
\begin{rem}[Choice of the set $\mathcal{Z}$]
  Other choices could have been made for $\mathcal{Z}$, however ours seems easier in terms of computation, since $|\mathcal{Z}|=p(p+1)/2$. Choosing a larger set $\mathcal{Z}$ leads to theoretical difficulties in the reconstruction of $\widehat{\S}$, while taking other basis vectors leads to more complex computations.
We can also note that increasing $|\mathcal{Z}|$ decreases the probability in Theorem~\ref{propmulti}, since it comes from an union bound over the one-dimensional estimations.
\end{rem}

\begin{rem}\label{rem_algo_simple}
  When $\Sh = \Sh_{\simple}$ as defined by \Eq{algo_simplifie}, that is, when a diagonalization basis of $\S$ is known, Theorem \ref{propmulti} still holds on a set of larger  probability $1- \kappa p n^{-\d}$ with a reduced error $\eta = L_1(\a+\d) \sqrt{{\ln(n)}/{n}}$. Then, a consistent estimation of $\S$ is possible whenever $p = O(n^{\d})$ for some $\d \geq 0$.
\end{rem}

\section{Oracle Inequality \label{ineg_oracle}}
This section aims at proving ``oracle inequalities'', as usually done in a model selection setting: given a set of models or of estimators, the goal is to upper bound the risk of the selected estimator by the oracle risk (defined by \Eq{oracle_risk}), up to an additive term and a multiplicative factor. We show two oracle inequalities (Theorems~\ref{thm_oracle_disc} and \ref{thm_oracle_HM}) that correspond to two possible definitions of $\Sh$.

Note that ``oracle inequality'' sometimes has a different meaning in the literature \citep[see for instance][]{lounici2011oracle} when the risk of the proposed estimator is controlled by the risk of an estimator using information coming from the true parameter (that is, available only if provided by an oracle).

\subsection{A General Result for Discrete Matrix Sets $\M$}
We first show that the estimator introduced in \Eq{def_Sigma} is precise enough to derive an oracle inequality when plugged in the penalty defined in \Eq{defpenidmulti} in the case where $\M$ is finite.

\begin{definition}
  Let $\widehat{\S}$ be the estimator of $\S$ defined by \Eq{def_Sigma}. We define
$$
    \Mh \in \argmin{M \in \M} \set{ \norr{ \fh_{M}-y}^2 +2\tr \paren{ A_{M}\cdot(\widehat{\S}\otimes I_n) } } \pt
$$
\end{definition}
We assume now the following   holds true:
\begin{equation}\label{Hdisc}
  \exists (C,\a_{\M}) \in (0,+\infty)^2, \quad\card(\M) < Cn^{\a_{\M}}\pt
\end{equation}

\begin{theorem} \label{thm_oracle_disc}
Let $\a = \max(\a_{\M},2)$, $\d \geq 2$ and assume~\eqref{Hdf} and~\eqref{Hdisc} hold true. Absolute constants $L_2, \kappa'>0$, a constant $n_1(\d)$ and an event $\tilde{\Omega}$ exist such that $\Proba(\tilde{\Omega}) \geq 1-\kappa 'p(p+C) n^{-\d}$
and the following holds as soon as $n \geq n_1(\d)$.
First, on $\tilde{\Omega}$,
\begin{equation} \label{resultatoracle_ht_proba}
  \begin{split}
  \frac{1}{np}\norr{\fh_{\Mh}-f}^2 \leq \left( 1+\frac{1}{\ln(n)} \right)^2 \inf_{M \in \M}\set{\frac{1}{np} \norr{\fh_{M}-f}^2 } + L_2  c(\S)^4\tr(\S) (\a + \d)^2\frac{p^3\ln(n)^3}{np}\pt
  \end{split}
\end{equation}
Second, an absolute constant $L_3$ exists such that
\begin{equation} \label{resultatoracle_esp}
  \begin{split}
  \esp{\frac{1}{np}\norr{\fh_{\Mh}-f}^2} \leq \left( 1+\frac{1}{\ln(n)} \right)^2 \esp{\inf_{M \in \M}\left\{\frac{1}{np} \norr{ \fh_{M}-f}^2  \right\}} ~~~~~~~~~~~~~~~~~~~~~~~\\~~~~~~~~~~~~+L_2 c(\S)^4\tr(\S) (\a + \d)^2\frac{p^3\ln(n)^3}{np}+ L_3\frac{\sqrt{p(p+C)}}{n^{\d/2}} \left(\nor{\S} + \frac{\norr{f}^2  }{n p} \right)
   \pt
  \end{split}
\end{equation}
\end{theorem}
Theorem~\ref{thm_oracle_disc} is proved in Section~\ref{proof_thmoracle}.

\begin{rem}
If $\Sh = \Sh_{\simple}$ is defined by \Eq{algo_simplifie} the result still holds on a set of larger probability $ 1-\kappa 'p(1+C) n^{-\d}$ with a reduced error, similar to the one in Theorem \ref{thm_oracle_HM}.
\end{rem}

\subsection{A Result for a Continuous Set of Jointly Diagonalizable Matrices \label{result_diagonalizable}}

We now show a similar result when matrices in $\M$ can be jointly diagonalized. It turns out  a faster algorithm can be used instead of \Eq{def_Sigma} with a reduced error and a larger probability event in the oracle inequality.  Note that we no longer assume $\M$ is finite, so it can be parametrized by continuous parameters.

Suppose now the following holds, which means the matrices of $\M$ are jointly diagonalizable:

\begin{equation}\label{HM}
  \exists P \in O_{p}(\R) \, , \quad \M \subseteq \set{ P\trsp \diag(d_1,\dots,d_p) P \, , \, (d_i)_{i=1}^p \in (0,+\infty)^p } \pt
\end{equation}

Let $P$ be the matrix defined in Assumption \eqref{HM}, $\tilde{\S} = P\Sigma P\trsp$ and recall that $A_{\l}=K(K+n\l I_n)^{-1}\,$. Computations detailed in Appendix \ref{app_calcul_oracle}  show that the ideal penalty introduced in \Eq{defpenidmulti} can be written as
\begin{equation} \label{pen_id_HM}
  \begin{split}
    \forall M=P\trsp\diag(d_1,\dots,d_p)P \in \M, ~~~~~~~~~~~~~~~~~~~~~~~~~~~~~~~~~~~~~~~~~~~~~~~~~~~~~~\\~~~~~~~~~~~~ \pen_{\id}(M) = \frac{2\tr \big( A_{M}\cdot(\S \otimes I_n) \big) }{np} = \frac{2}{np} \left(\sum_{j=1}^p \tr(A_{pd_j})\tilde{\S}_{j,j}  \right) \pt
  \end{split}
\end{equation}

\Eq{pen_id_HM} shows that under Assumption \eqref{HM}, we do not need to estimate the entire matrix $\S$ in order to have a good penalization procedure, but only to estimate the variance of the noise in $p$ directions.

\begin{definition} \label{def_Sh_HM}
  Let $(e_1,\dots,e_p)$ be the canonical basis of $\R^p$, $(u_1,\dots,u_p)$ be the orthogonal basis defined by $\forall j \in \{1,\dots,p\},~ u_j=P\trsp e_j$.
  We then define
\begin{equation*}
\hat{\S}_{\textrm{HM}} = P \diag(a(u_1),\dots,a(u_p))P\trsp \virg
\end{equation*}
where for every $j \in \{1,\dots,p\}$, $a(u_j)$ denotes the  output  of Algorithm \ref{algo1} applied to Problem ($\mathbf{Pu_{j}}$), and
  \begin{equation} \label{deflc_HM}
    \Mh_{\textrm{HM}} \in \argmin{M \in \M} \set{ \norr{ \fh_{M}-y}^2 +2\tr \paren{ A_{M}\cdot(\hat{\S}_{\textrm{HM}}\otimes I_n) } } \pt
  \end{equation}
\end{definition}

\begin{theorem} \label{thm_oracle_HM}
Let $\a = 2$, $\d \geq 2$ and assume~\eqref{Hdf} and~\eqref{HM} hold true. Absolute constants $L_2>0$, and $\kappa''$, a constant $n_1(\d)$ and an event $\tilde{\Omega}$ exist such that $\Proba(\tilde{\Omega}) \geq 1-\kappa''p n^{-\d}$
and the following holds as soon as $n \geq n_1(\d)$.
 First, on $\tilde{\Omega}$,
\begin{equation} \label{resultatoracle_ht_proba_HM}
  \begin{split}
  \frac{1}{np}\norr{\fh_{\Mh_{\textrm{HM}}}-f}^2 \leq \left( 1+\frac{1}{\ln(n)} \right)^2 \inf_{M \in \M}\set{\frac{1}{np} \norr{\fh_{M}-f}^2 } + L_2 \tr(\S) (2 + \d)^2\frac{\ln(n)^3}{n}\pt
  \end{split}
\end{equation}
Second, an absolute constant $L_4$ exists such that
\begin{equation} \label{resultatoracle_esp_HM}
  \begin{split}
  \esp{\frac{1}{np}\norr{\fh_{\Mh_{\textrm{HM}}}-f}^2} \leq \left( 1+\frac{1}{\ln(n)} \right)^2 \esp{\inf_{M \in \M}\left\{\frac{1}{np} \norr{ \fh_{M}-f}^2  \right\}} \\+L_4\tr(\S) (2 + \d)^2\frac{\ln(n)^3}{n}
 + \frac{p}{n^{\d/2}} \frac{\norr{f}^2  }{n p}
   \pt
  \end{split}
\end{equation}
\end{theorem}
Theorem~\ref{thm_oracle_HM} is proved in Section~\ref{proof_thmoracle}.

\subsection{Comments on Theorems \ref{thm_oracle_disc} and \ref{thm_oracle_HM}}
\begin{rem} Taking $p=1$ \textup{(}hence $c(\S)=1$ and $\tr(\Sigma)=\sigma^2\,$\textup{)}, we recover Theorem~3 of \citet{Arl_Bac:2009:minikernel_long} as a corollary of Theorem~\ref{thm_oracle_disc}. \end{rem}

\begin{rem}[Scaling of $(n,p)$]
When assumption \eqref{Hdisc} holds, \Eq{resultatoracle_ht_proba} implies the asymptotic optimality of the estimator $\fh_{\Mh}$ when
\begin{equation*}
c(\Sigma)^4 \frac{ \tr{\S}}{p}  \times \frac{ p^3 \paren{\ln(n)}^3}{n} \ll \inf_{M \in \M}\set{\frac{1}{np} \norr{\fh_{M}-f}^2 } \pt
\end{equation*}
In particular, only $(n,p)$ such that $p^3 \ll n / \sparen{\ln(n)}^3$ are admissible.
When assumption \eqref{HM} holds, the scalings required to ensure optimality in \Eq{resultatoracle_ht_proba_HM} are more favorable:
\begin{equation*}
\tr{\S}  \times \frac{ \paren{\ln(n)}^3}{n} \ll \inf_{M \in \M}\set{\frac{1}{np} \norr{\fh_{M}-f}^2 } \pt
\end{equation*}
It is to be noted that $p$ still influences the left hand side via $\tr{\S}$.
\end{rem}

\begin{rem}
Theorems \ref{thm_oracle_disc} and \ref{thm_oracle_HM} are non asymptotic oracle inequalities, with a multiplicative term of the form $1+o(1)$. This allows us to claim that our selection procedure is nearly optimal, since our estimator is close \textup{(}with regard to the empirical quadratic norm\textup{)} to the oracle one.  Furthermore the term $1+(\ln(n))^{-1}$ in front of the infima in Equations~\eqref{resultatoracle_ht_proba}, \eqref{resultatoracle_ht_proba_HM}, \eqref{resultatoracle_esp} and \eqref{resultatoracle_esp_HM}  can be further diminished, but this yields a greater remainder term as a consequence.
\end{rem}

\begin{rem}[On assumption~\eqref{HM}]\label{rem_HM}
Assumption~\eqref{HM} actually means all matrices in $\M$ can be diagonalized in a unique orthogonal basis, and thus can be parametrized by their eigenvalues as in Examples \ref{example_ind}, \ref{example_Mlm} and \ref{example_clustering}.

In that case the optimization problem is quite easy to solve, as detailed in Remark \ref{rem_easy_optim_HM}. If not, solving \eqref{deflc_HM} may turn out to be a hard problem, and our theoretical results do not cover this setting. However, it is always possible to discretize the set $\M$ or, in practice, to use gradient descent.

Compared to the setting of Theorem \ref{thm_oracle_disc}, assumption \eqref{HM} allows a simpler estimator for the penalty \eqref{pen_id_HM}, with an increased probability  and a reduced error in the oracle inequality.

The main theoretical limitation comes from the fact that the probabilistic concentration tools used apply to discrete sets $\M$ (through union bounds). The structure of kernel ridge regression allows us to have a uniform control over a continuous set for the single-task estimators at the ``cost'' of $n$ pointwise controls, which can then be extended to the multi-task setting via \eqref{HM}. We conjecture Theorem~\ref{thm_oracle_HM} still holds without \eqref{HM} as long as $\M$ is not ``too large'', which could be proved similarly up to some uniform concentration inequalities.

Note also that if $\M_1,\ldots, \M_K$ all satisfy \eqref{HM} (with different matrices $P_k$), then Theorem~\ref{thm_oracle_HM} still holds for $\M = \bigcup_{k=1}^K \M_k$ with the penalty defined  by \Eq{deflc_HM} with $P = P_k$ when $M \in \M_k$, and $\Proba(\tilde{\Omega}) \geq 1- 9 K p^2 n^{-\d}$, by applying the union bound in the proof.
\end{rem}

\begin{rem}[Relationship with the trace norm]
Our approach relies on the minimization of Equation~\eqref{min_multi_M} with respect to $f$. \citet{DBLP:journals/ml/ArgyriouEP08} has shown that if we also minimize Equation~\eqref{min_multi_M} with respect to the matrix $M$ subject to the constraint $ \tr M^{-1}=1$, then we obtain an equivalent regularization by the nuclear norm \textup{(}a.k.a. trace norm\textup{)}, which implies the prior knowledge that our~$p$ prediction functions may be obtained as the linear combination of $ r \ll p$ basis functions. This situation corresponds to cases where the matrix $M^{-1}$ is singular.

Note that the link between our framework and trace norm \textup{(}i.e., nuclear norm\textup{)} regularization is the same than between multiple kernel learning and the single task framework of \citet{Arl_Bac:2009:minikernel_long}. In the multi-task case, the trace-norm regularization, though efficient computationally, does not lead to an oracle inequality, while our criterion is an unbiased estimate of the generalization error, which turns out to be non-convex in the matrix $M$. While DC programming techniques \citep[see, e.g.,][and references therein]{gasso2009} could be brought to bear to find local optima, the goal of the present work is to study the theoretical properties of our estimators, assuming we can minimize the cost function \textup{(}e.g., in special cases, where we consider spectral variants, or by brute force enumeration\textup{)}.
\end{rem}

\section{Simulation Experiments \label{simulations}}
In all the experiments presented in this section, we consider the framework of Section~\ref{sec_anal_th} with $\X = \R^d$, $d = 4$, and the kernel defined by $\forall x,y \in \X$, $k(x,y) =  \prod_{j=1}^d e^{-|x_j - y_j|}$.
The design points $X_1, \ldots, X_n \in \R^d$ are drawn (repeatedly and independently for each sample) independently from the multivariate standard Gaussian distribution.
For every $j \in \sset{1, \ldots, p}$, $f^j(\cdot) = \sum_{i=1}^m \a_i^j k(\cdot,z_i)$ where $m = 4$ and $z_1, \ldots, z_m \in \R^d$ are drawn (once for all experiments except in Experiment~D) independently from the multivariate standard Gaussian distribution, independently from the design $(X_i)_{1 \leq i \leq n}$. Thus, the expectations that will be considered are taken conditionally to the $z_i$.
The coefficients $\sparen{\a_i^j}_{1 \leq i \leq m \, , \, 1 \leq j \leq p}$ differ according to the setting.
Matlab code is available online.\footnote{Matlab code can be found at \url{http://www.di.ens.fr/~solnon/multitask_minpen_en.html}.}
\subsection{Experiments} Five experimental settings are considered:
\begin{enumerate}
\item[{\bf A$\rfloor$}] {\bf Various numbers of tasks:} $n = 10$ and $\forall i,j$, $\a_{i}^j = 1$, that is, $\forall j$, $f^j = f_A \egaldef \sum_{i=1}^m k(\cdot,z_i) $.
The number of tasks is varying: $p \in \sset{2k \, / \, k=1, \ldots, 25}$.
The covariance matrix is $\S = 10\cdot I_p$.
\item[{\bf B$\rfloor$}] {\bf Various sample sizes:} $p = 5$, $\forall j$, $f^j = f_A$ and $\S= \S_{B}$ has been drawn (once for all) from the Whishart $W(I_5,10,5)$ distribution; the condition number of $\S_{B}$ is $c(\S_{B}) \approx 22.05$.
The only varying parameter is $n \in \sset{50k \, / \, k=1, \ldots, 20}$.
\item[{\bf C$\rfloor$}] {\bf Various noise levels:} $n=100$, $p=5$ and $\forall j$, $f^j = f_A\,$. The varying parameter is $\S = \S_{C,t} \egaldef 5t \cdot I_5$ with $t \in \set{0.2 k \, / \, k=1, \ldots, 50}$. We also ran the experiments for $t = 0.01$ and $t = 100$.
\item[{\bf D$\rfloor$}] {\bf Clustering of two groups of functions:} $p = 10$, $n=100$, $\S=\S_{E}$  has been drawn (once for all) from the Whishart $W(I_{10},20,10)$ distribution; the condition number of $\S_{E}$ is $c(\S_{E}) \approx 24.95$. We pick the function $f_D \egaldef \sum_{i=1}^m \a_i k(\cdot,z_i)$ by drawing $(\a_1,\dots,\a_m)$ and $(z_1,\dots,z_m)$ from standard multivariate normal distribution (independently in each replication) and finally $f^1=\dots=f^5=f_D$, $f^6=\dots=f^{10}=-f_D$.
\item[{\bf E$\rfloor$}] {\bf Comparison to cross-validation parameter selection:} $p=5$, $\S = 10\cdot I_5$, $\forall j$, $f^j = f_A$. The sample size is taken in $\{10,50,100,250\}$.
\end{enumerate}
\subsection{Collections of Matrices} Two different sets of matrices $\M$ are considered in the Experiments A--C, following Examples~\ref{example_ind} and~\ref{example_Mlm}:
\begin{align*}
\Mmultiens &\egaldef \set{ \Mmulti \paren{\l,\mu} = (\l+p\mu) I_p - \frac{\mu}{p} \boldsymbol{1}\boldsymbol{1}\trsp \, / \, (\l,\mu) \in (0,+\infty)^2 } \\
\mbox{and} \quad \M_{\singletask} &\egaldef \set{M_{\ind}(\l) = \diag(\l_1,\dots,\l_p) \, / \, \l \in (0,+\infty)^p } \pt
\end{align*}
In Experiment D, we also use two different sets of matrices, following Example~\ref{example_clustering}:
\begin{align*}
\Mclustering &\egaldef \bigcup_{ I \subset \{1,\dots,p\}, I \notin\set{\{1,\dots,p\},\emptyset}} \set{ \MI \paren{\l,\mu,\mu}   \, / \, (\l,\mu) \in (0,+\infty)^2 }\cup\Mmultiens \\
\mbox{and} \quad \Msegmentation &\egaldef \bigcup_{ 1\leq k \leq p-1} \set{ \MI \paren{\l,\mu,\mu}   \, / \, (\l,\mu) \in (0,+\infty)^2, I = \{1,\dots,k\} }\cup\Mmultiens \pt
\end{align*}
\begin{rem}
The set $\Mclustering$ contains $2^p-1$ models, a case we will denote by ``clustering''. The other set, $\Msegmentation$, only has $p$ models, and is adapted to the structure of the Experiment D. We call this setting ``segmentation into intervals''.
\end{rem}

\subsection{Estimators} In Experiments A--C, we consider four estimators obtained by combining two collections $\M$ of matrices with two formulas for $\S$ which are plugged into the penalty \eqref{defpenidmulti} (that is, either $\S$ known or estimated by $\widehat{\S}$):
\begin{gather*}
\forall \alpha   \in \set{\multitask , \singletask } \, , \,  \forall S \in \set{\S,\Sh_{\HM}} \, , \quad
\fh_{\alpha,S} \egaldef \fh_{\Mh_{\alpha,S}} = A_{\Mh_{\alpha,S}} y
\\
\mbox{where} \quad
\Mh_{\alpha,S} \in \argmin{M \in \M_{\alpha}} \set{ \frac{1}{np} \norm{ y - \fh_{M} }_2^2 + \frac{2}{np}\tr\paren{ A_M\cdot (S\otimes I_n )}}
\end{gather*}
and $\Sh_{\HM}$ is defined in Section~\ref{result_diagonalizable}.
As detailed in Examples~\ref{example_ind}--\ref{example_Mlm}, $\fh_{\singletask,\Sh_{\HM}}$ and $\fh_{\singletask,\S}$ are concatenations of single-task estimators, whereas $\fh_{\multitask,\widehat{\S}_{\HM}}$ and $\fh_{\multitask,\S}$ should take advantage of a setting where the functions $f^j$ are close in $\F$ thanks to the regularization term $\sum_{j,k}\|f^j-f^k\|_{\F}^2$.
In Experiment D we consider the following three estimators, that depend on the choice of the collection $\M$:
\begin{gather*}
\forall \b   \in \set{\clustering , \segmentation, \singletask } \,  , \quad
\fh_{\b} \egaldef \fh_{\Mh_{\b}} = A_{\Mh_{\b}} y
\\
\mbox{where} \quad
\Mh_{\b} \in \argmin{M \in \M_{\b}} \set{ \frac{1}{np} \norm{ y - \fh_{M} }_2^2 + \frac{2}{np}\tr\paren{ A_M\cdot (\Sh\otimes I_n )}}
\end{gather*}
and $\Sh$ is defined by Equation~\eqref{def_Sigma}.

In Experiment E we consider the estimator $\fh_{\multitask,\Sh_{\HM}}$. As explained in the following remark the parameters of the former estimator are chosen by optimizing \eqref{deflc_HM}, in practice by choosing a grid. We also consider the estimator $\fh_{\multitask,\textrm{CV}}$ where the parameters are selected by performing 5-fold cross-validation on the mentionned grid.

\begin{rem}[Optimization of \eqref{deflc_HM}]\label{rem_easy_optim_HM}
  Thanks to Assumption \eqref{HM} the optimization problem \eqref{deflc_HM} can be solved easily. It suffices to diagonalize in a common basis the elements of $\M$ and the problem splits into several multi-task problems, each with one real parameter. The optimization was then done by using a grid on the real parameters, chosen such that the degree of freedom takes all integer values from $0$ to $n$.
\end{rem}

\begin{rem}[Finding the jump in Algorithm \ref{algo1}]
Algorithm~\ref{algo1} raises the question of how to detect the jump of  $\df(\l)$, which happens around $C=\s^2$. We chose to select an estimator $\Ch$ of $\s^2$ corresponding to the smallest index such that  $\df(\widehat{\l}_0(\Ch)) < n/2$. Another approach is to choose the index corresponding to the largest instantaneous jump of $\df(\widehat{\l}_0(C))$ \textup{(}which is piece-wise constant and non-increasing\textup{)}. This approach has a major drawback, because it sometimes selects a jump far away from the ``real'' jump around $\s^2$, when the real jump consists of several small jumps. Both approaches gave similar results in terms of prediction error, and we chose the first one because of its direct link to the theoretical criterion given in Theorem~\ref{thmono}.
\end{rem}

\subsection{Results} In each experiment, $N=1000$ independent samples $y \in \R^{np}$ have been generated. Expectations are estimated thanks to empirical means over the $N$ samples. Error bars correspond to the classical Gaussian $95\%$ confidence interval (that is, empirical standard-deviation over the $N$ samples multiplied by $1.96 / \sqrt{N}$).
The results of Experiments A--C are reported in Figures~\ref{fig_variation_p_2}--\ref{fig_variation_sigma_1}.
The results of Experiments C--E are reported in Tables~\ref{table_variation_sigma}--\ref{table_cv}.
The p-values correspond to the classical Gaussian difference test, where the hypotheses tested are of the shape $\mathbb{H}_0 = \set{q > 1}$ against the hypotheses $\mathbb{H}_1 = \set{q \leq 1}$, where the different quantities $q$ are detailed in Tables \ref{table_clustering}--\ref{table_cv}.

\begin{figure}
\centering
\includegraphics[width=\textwidth]{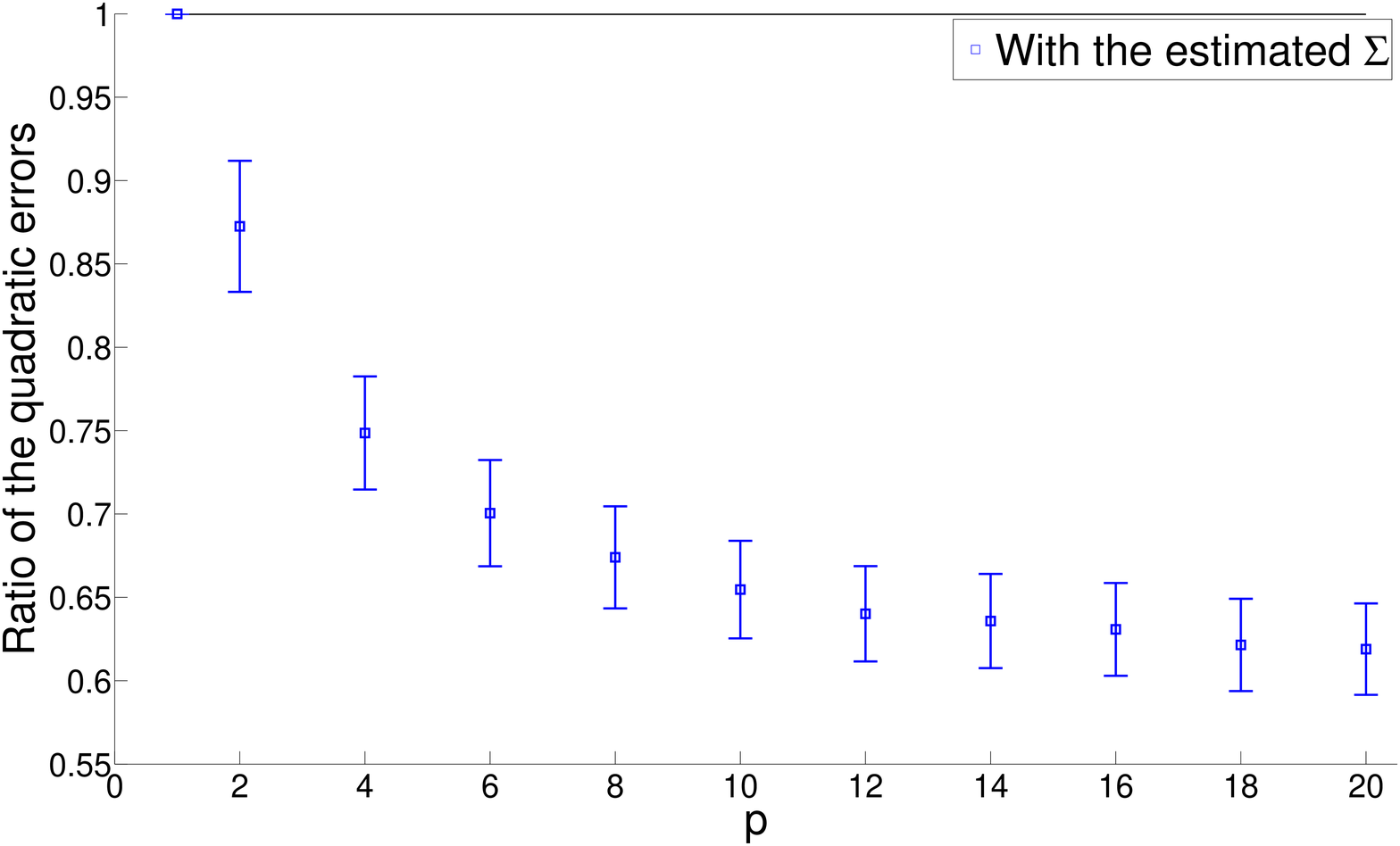}
\caption{Increasing the number of tasks $p$ (Experiment~A), improvement of multi-task compared to single-task: $\E\scroch{ \snorm{\fh_{\multitask,\widehat{\S}} - f}^2 / \snorm{ \fh_{\singletask,\widehat{\S}} -f}^2 }$.
}
\label{fig_variation_p_1}
\end{figure}

\begin{figure}
\centering
\includegraphics[width=\textwidth]{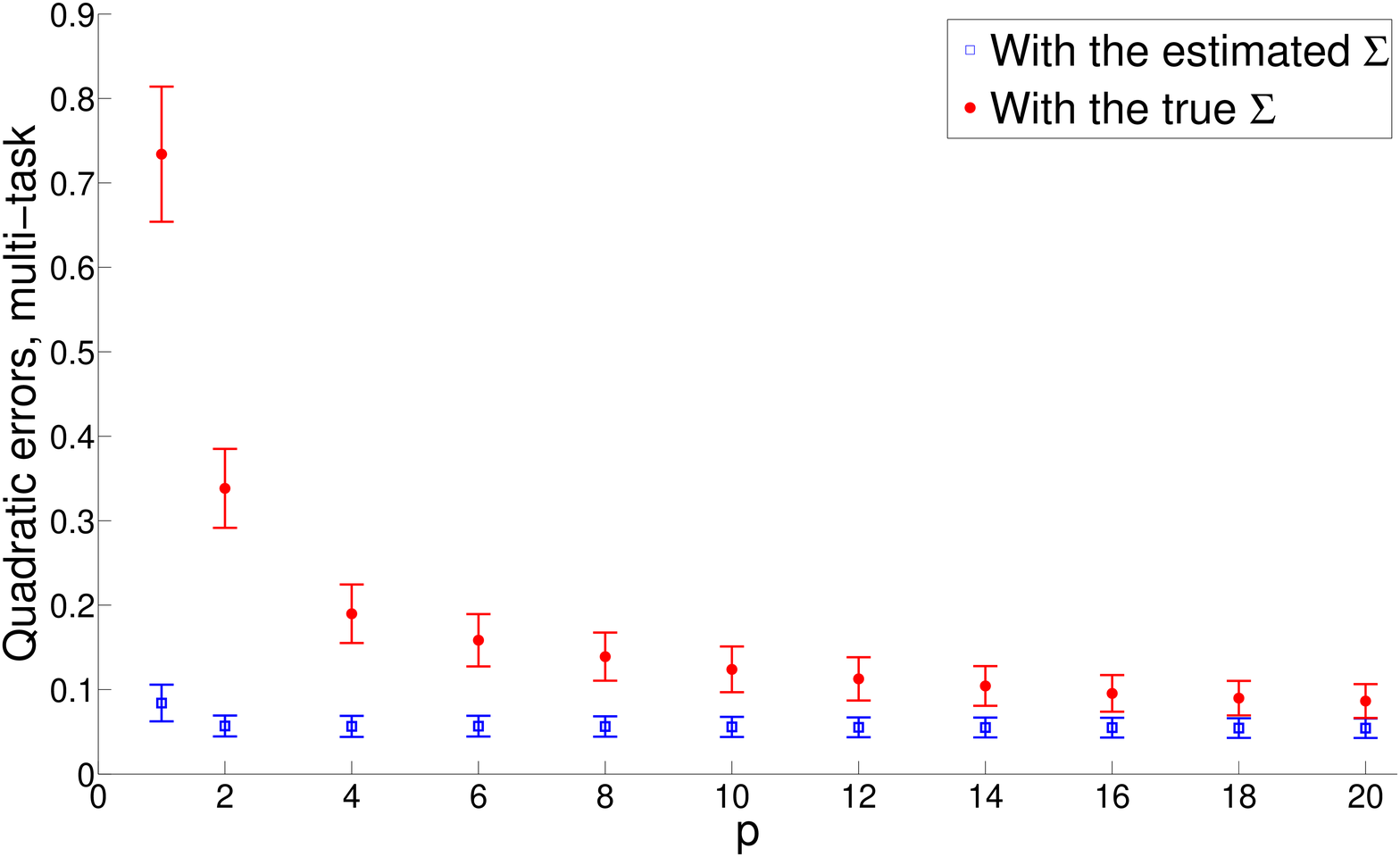}
\caption{Increasing the number of tasks $p$ (Experiment~A), quadratic errors of multi-task estimators $(np)^{-1}\E\scroch{\snorm{\fh_{\multitask,S}-f}^2}$.
Blue:  $S=\Sh$. 
Red: $S=\S$. 
}
\label{fig_variation_p_2}
\end{figure}

\begin{figure}
\centering
\includegraphics[width=\textwidth]{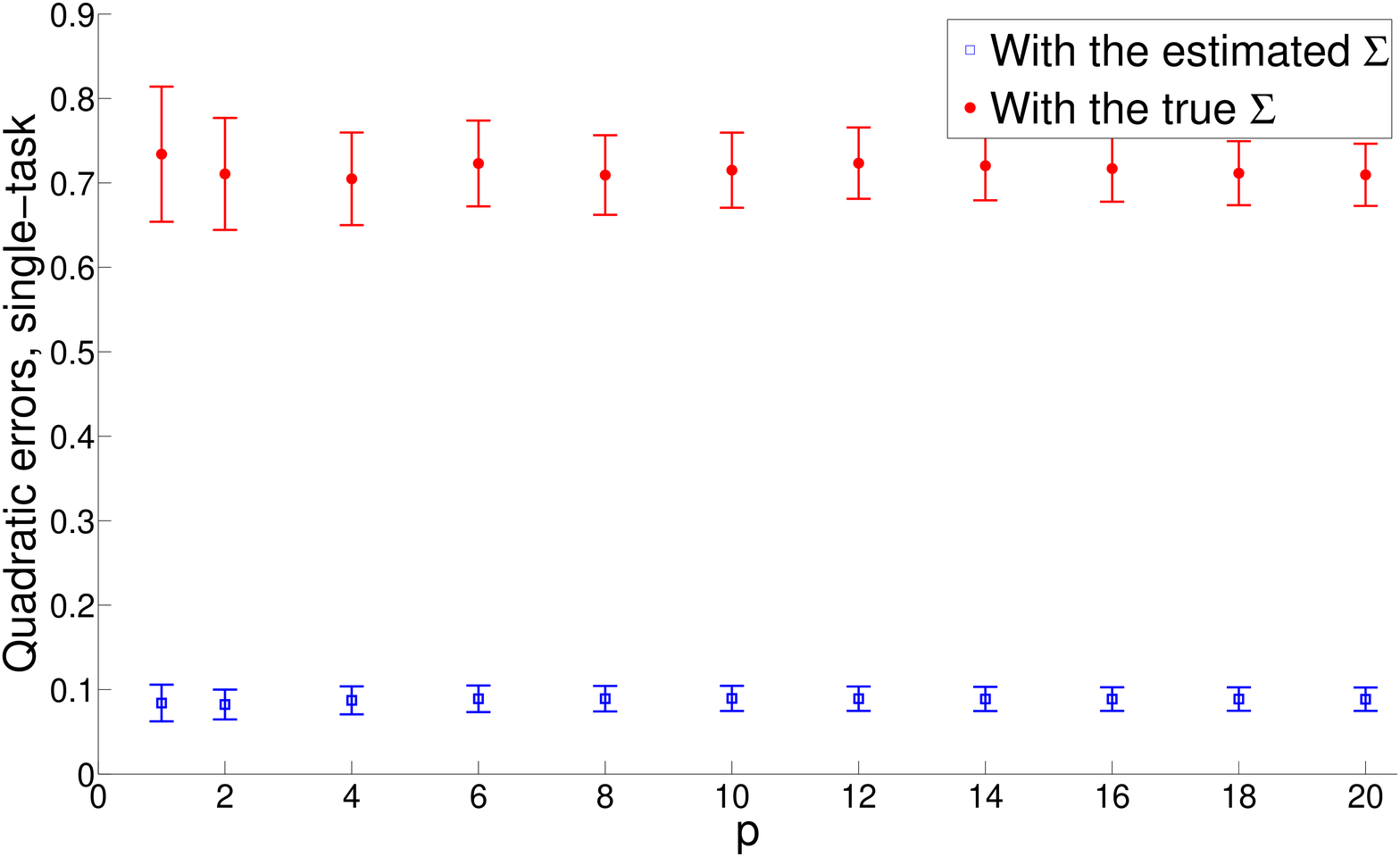}
\caption{Increasing the number of tasks $p$ (Experiment~A), quadratic errors of single-task estimators $(np)^{-1}\E\scroch{\snorm{\fh_{\singletask,S}-f}^2}$.
Blue:  $S=\Sh$. 
Red: $S=\S$. 
}
\label{fig_variation_p_3}
\end{figure}

\begin{figure}
\centering
\includegraphics[width=\textwidth]{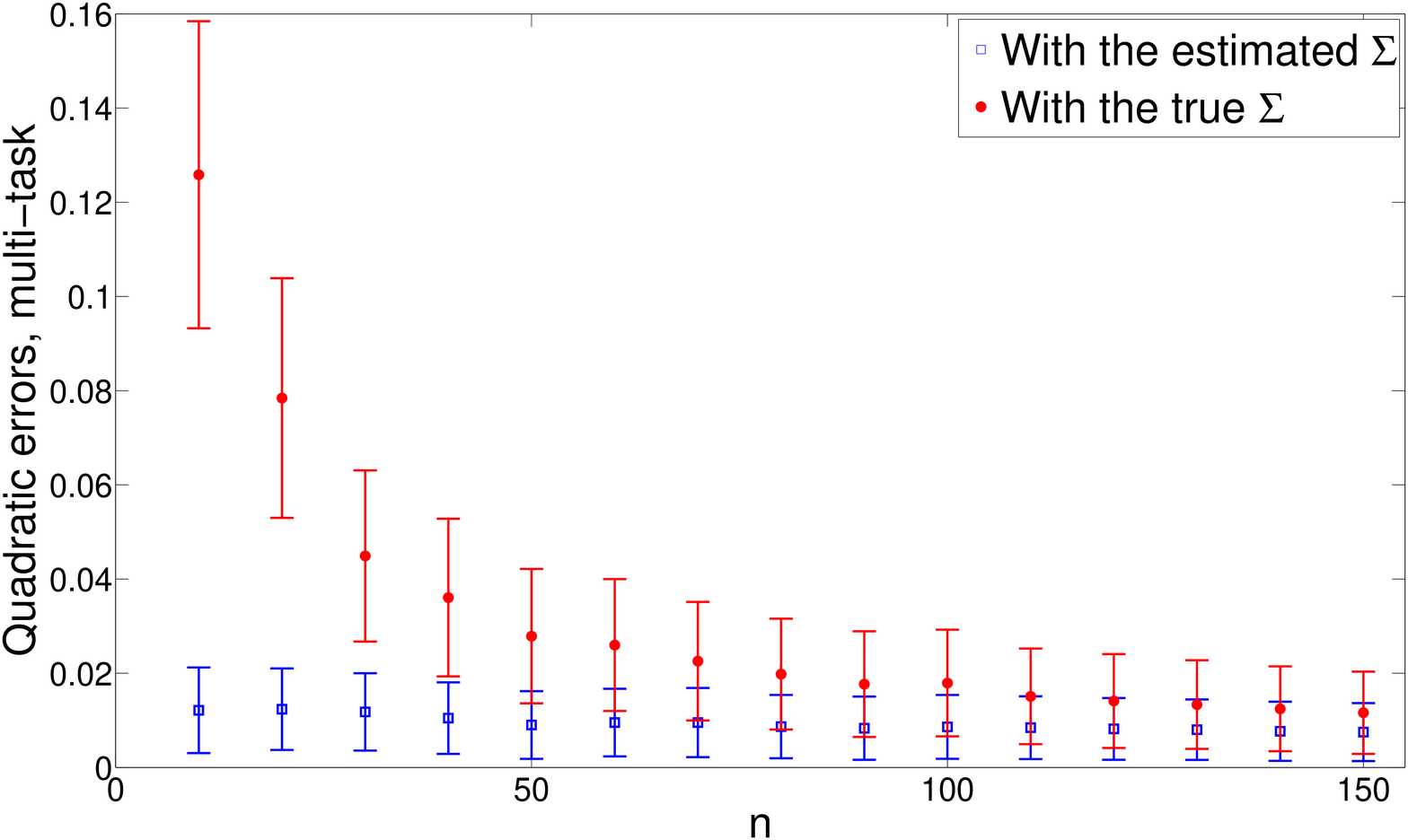}
\caption{Increasing the sample size $n$ (Experiment~B), quadratic errors of multi-task estimators $(np)^{-1}\E\scroch{\snorm{\fh_{\multitask,S}-f}^2}$.
Blue:  $S=\Sh$. 
Red: $S=\S$. 
}
\label{fig_variation_n_2}
\end{figure}

\begin{figure}
\centering
\includegraphics[width=\textwidth]{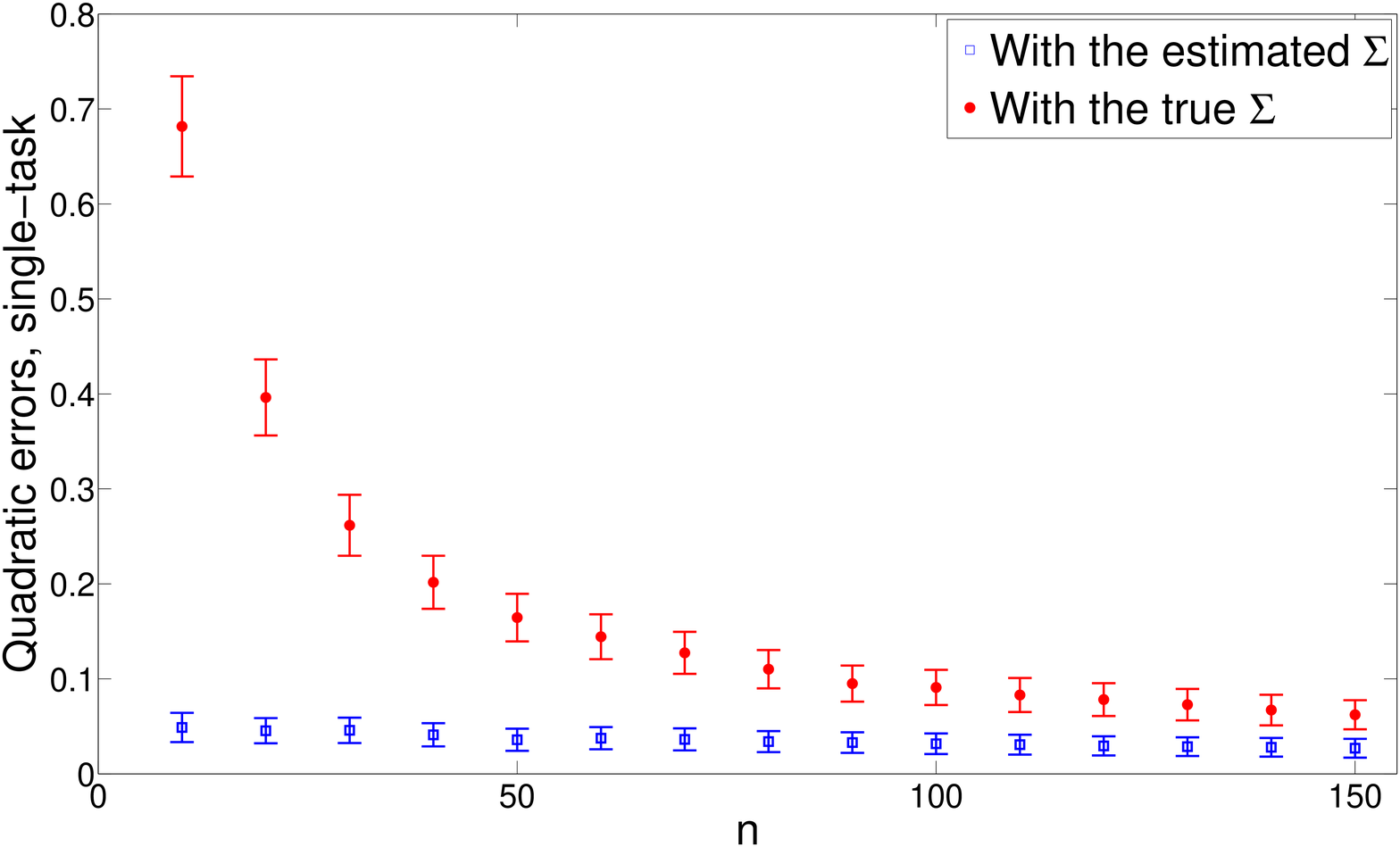}
\caption{Increasing the sample size $n$ (Experiment~B), quadratic errors of single-task estimators $(np)^{-1}\E\scroch{\snorm{\fh_{\singletask,S}-f}^2}$.
Blue:  $S=\Sh$. 
Red: $S=\S$. 
}
\label{fig_variation_n_3}
\end{figure}

\begin{figure}
\centering
\includegraphics[width=\textwidth]{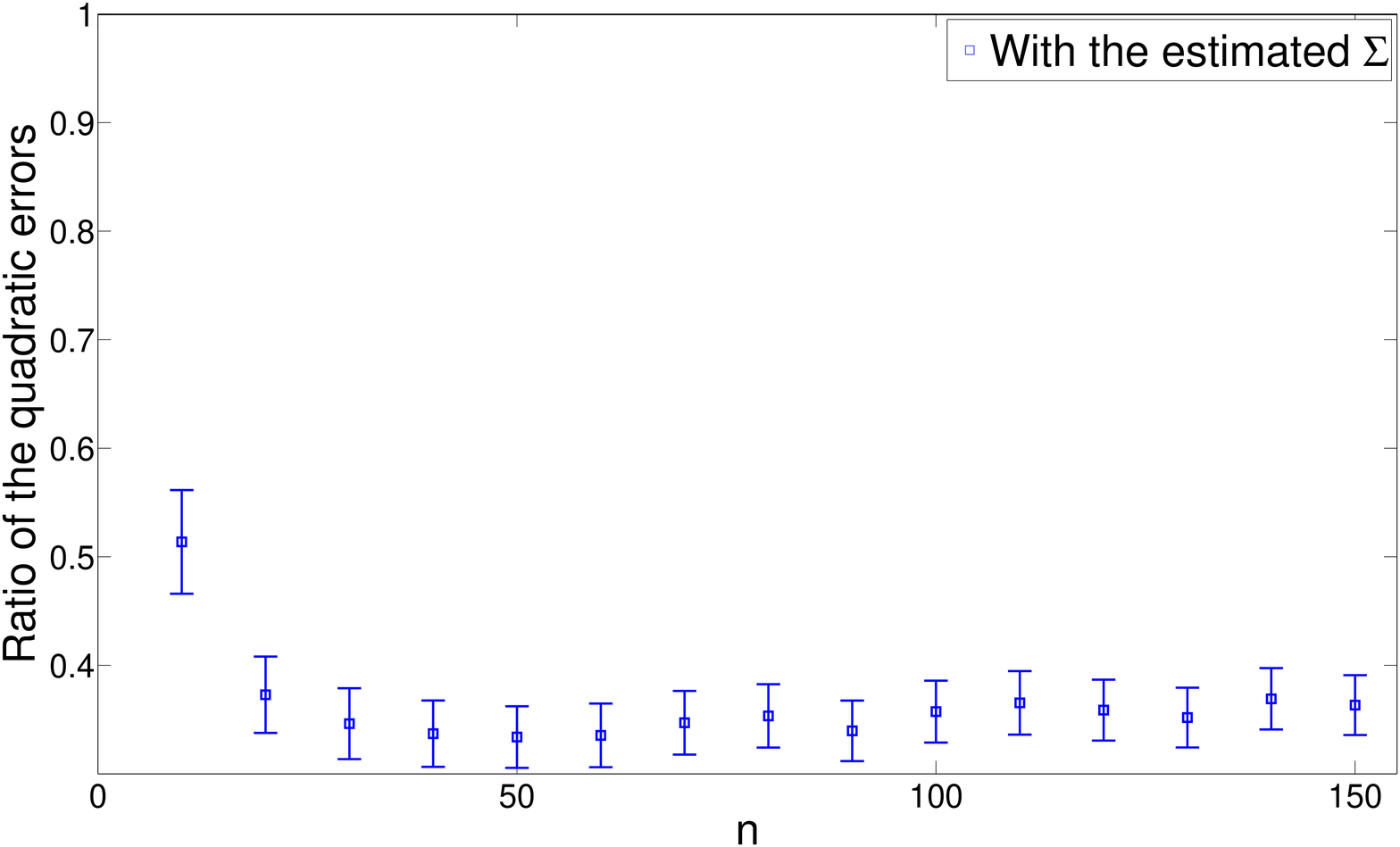}
\caption{Increasing the sample size $n$ (Experiment~B), improvement of multi-task compared to single-task:
$\E\scroch{ \snorm{\fh_{\multitask,\widehat{\S}} - f}^2 / \snorm{ \fh_{\singletask,\widehat{\S}} -f}^2 }$.
}
\label{fig_variation_n_1}
\end{figure}

\begin{figure}
\centering
\includegraphics[width=\textwidth]{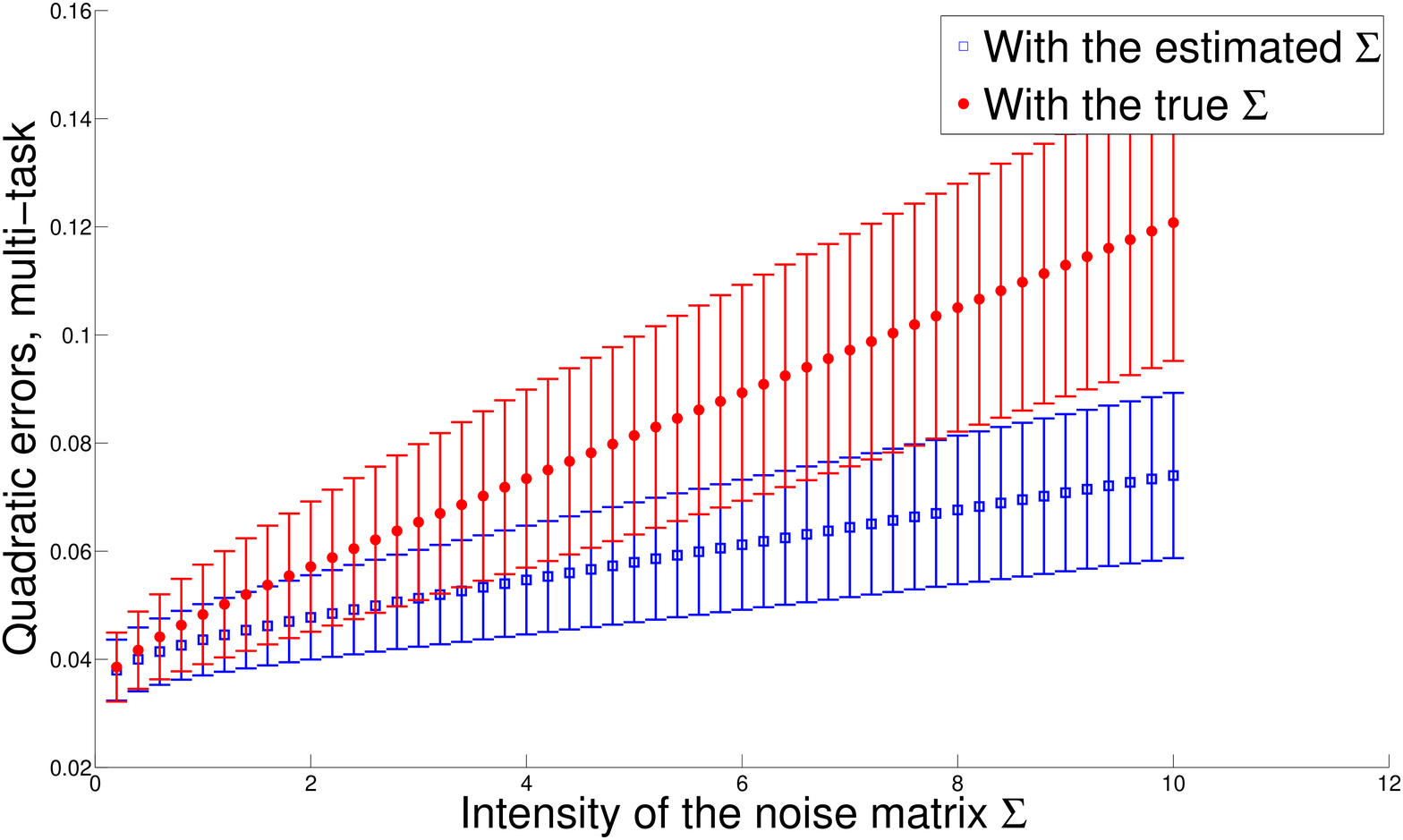}
\caption{Increasing the signal-to-noise ratio (Experiment~C), quadratic errors of multi-task estimators $(np)^{-1}\E\scroch{\snorm{\fh_{\multitask,S}-f}^2}$.
Blue:  $S=\Sh$. 
Red: $S=\S$. 
}
\label{fig_variation_sigma_2}
\end{figure}

\begin{figure}
\centering
\includegraphics[width=\textwidth]{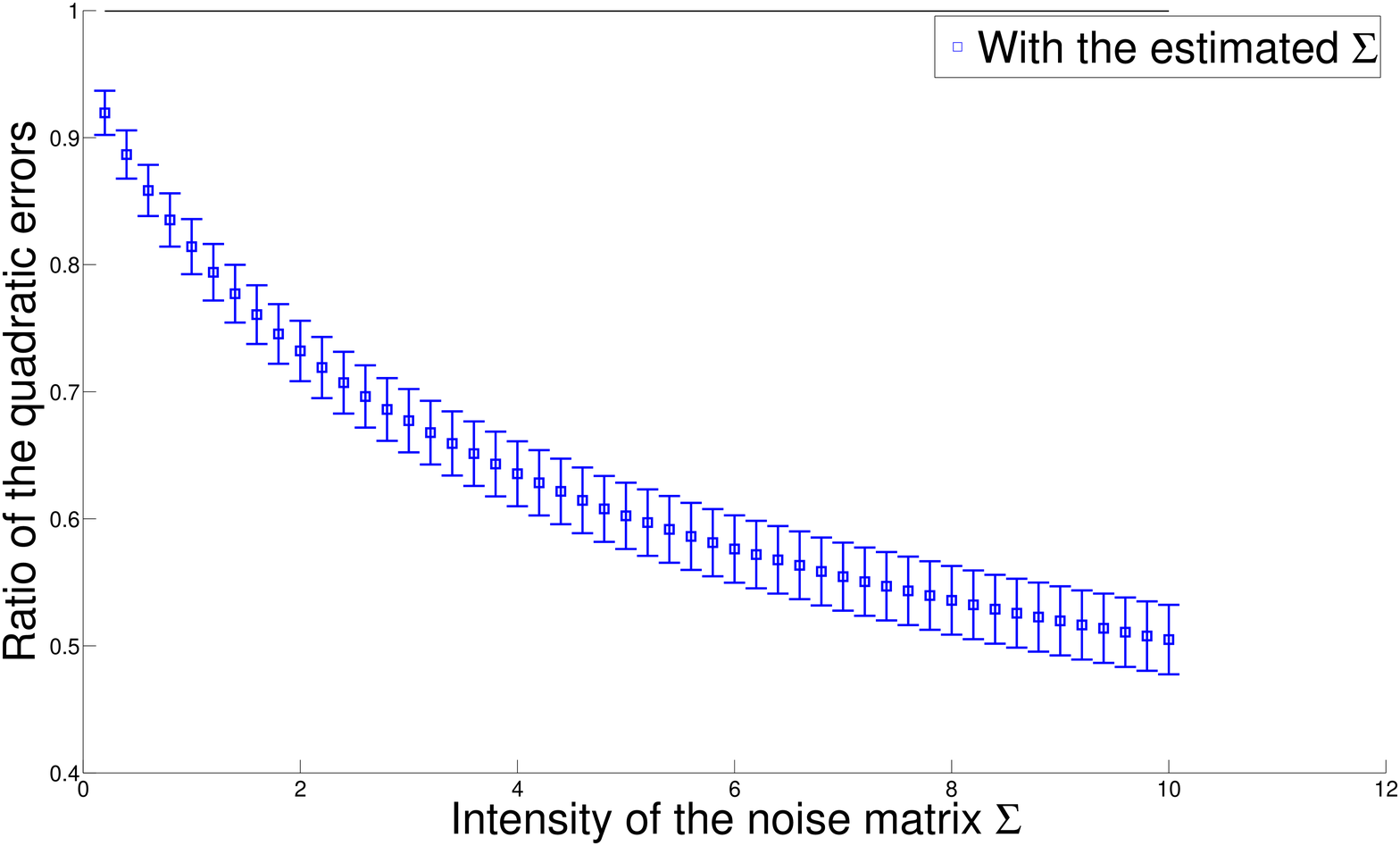}
\caption{Increasing the signal-to-noise ratio (Experiment~C), improvement of multi-task compared to single-task:
$\E\scroch{ \snorm{\fh_{\multitask,\widehat{\S}} - f}^2 / \snorm{ \fh_{\singletask,\widehat{\S}} -f}^2 }$.
}
\label{fig_variation_sigma_1}
\end{figure}

\begin{table}[!h]
\centering
\begin{tabular}{|c|c|c|}
\hline
 $t$ &0.01&100 \\
\hline
$\E\scroch{ \snorm{\fh_{\multitask,\Sh} - f}^2 / \snorm{ \fh_{\singletask,\Sh} -f}^2 }$ &$1.80\pm 0.02$ & $0.300 \pm 0.003$ \\
\hline
$\E\scroch{ \snorm{\fh_{\multitask,\Sh} - f}^2}$ &$(2.27 \pm 0.38) \times 10^{-2}$& $0.357 \pm 0.048$\\
\hline
$\E\scroch{ \snorm{\fh_{\multitask,\S} - f}^2  }$ &$(1.20 \pm 0.28) \times 10^{-2}$& $0.823 \pm 0.080 $\\
\hline
$\E\scroch{  \snorm{ \fh_{\singletask,\Sh} -f}^2 }$&$(1.26 \pm 0.26) \times 10^{-2}$& $ 1.51 \pm 0.07$\\
\hline
$\E\scroch{ \snorm{ \fh_{\singletask,\S} -f}^2 }$ &$(1.20 \pm 0.24) \times 10^{-2}$& $ 4.47 \pm 0.13$\\
\hline
\end{tabular}
\caption{Results of Experiment~C for the extreme values of $t$.}
\label{table_variation_sigma}
\end{table}

\begin{table}[!h]
\centering
\begin{tabular}{|c|c|c|c|}
\hline
$q$ & $\esp{q}$ & $\textrm{Std}[q]$ & p-value for $\mathbb{H}_0 = \set{q > 1}$\\
\hline
 $\snorm{\fh_{\clustering} - f}^2 / \snorm{ \fh_{\singletask} -f}^2 $  & $0.668 $& $0.294 $ &$<10^{-15}$\\
\hline
 $\snorm{\fh_{\segmentation} - f}^2 / \snorm{ \fh_{\singletask} -f}^2 $ &$0.660$  &$0.270$  &$<10^{-15}$ \\
\hline
 $\snorm{\fh_{\segmentation} - f}^2 / \snorm{ \fh_{\clustering} -f}^2$ & $1.00$& $0.165$  & $0.50$ \\
\hline
\end{tabular}
\caption{Clustering and segmentation (Experiment~D).}
\label{table_clustering}
\end{table}

\begin{table}[!h]
\centering
\begin{tabular}{|c|c|c|c|c|}
\hline
$q$ & $n$ &$\esp{q}$ &$\textrm{Std}[q]$  & p-value for $\mathbb{H}_0 = \set{q > 1}$ \\
\hline
$\snorm{\fh_{\multitask,\Sh_{\HM}} - f}^2 / \snorm{ \fh_{\multitask,\textrm{CV}} -f}^2$ & 10 &$0.35$ &$0.46$  &$<10^{-15}$\\
\hline
$ \snorm{\fh_{\multitask,\Sh_{\HM}} - f}^2 / \snorm{ \fh_{\multitask,\textrm{CV}} -f}^2$ & 50 &$0.56$ &$0.42$  &$<10^{-15}$\\
\hline
$ \snorm{\fh_{\multitask,\Sh_{\HM}} - f}^2 / \snorm{ \fh_{\multitask,\textrm{CV}} -f}^2$ & 100 &$0.71$ &$0.34$  &$<10^{-15}$\\
\hline
$ \snorm{\fh_{\multitask,\Sh_{\HM}} - f}^2 / \snorm{ \fh_{\multitask,\textrm{CV}} -f}^2$ & 250 &$0.87$& $0.19$  &$<10^{-15}$\\
\hline
\end{tabular}
\caption{Comparison of our method to 5-fold cross-validation (Experiment~E). }
\label{table_cv}
\end{table}

\subsection{Comments}
As expected, multi-task learning significantly helps when all $f^j$ are equal, as soon as $p$ is large enough (Figure~\ref{fig_variation_p_1}), especially for small $n$ (Figure~\ref{fig_variation_n_1}) and large noise-levels (Figure~\ref{fig_variation_sigma_1} and Table~\ref{table_variation_sigma}). Increasing the number of tasks rapidly reduces the quadratic error with multi-task estimators (Figure~\ref{fig_variation_p_2}) contrary to what happens with single-task estimators (Figure~\ref{fig_variation_p_3}).

A noticeable phenomenon also occurs in Figure~\ref{fig_variation_p_2} and even more in Figure~\ref{fig_variation_p_3}:
the estimator $\fh_{\singletask,\S}$ (that is, obtained knowing the true covariance matrix $\S$) is less efficient than $\fh_{\singletask,\Sh}$ where the covariance matrix is estimated.
It corresponds to the combination of two facts:
(i) multiplying the ideal penalty by a small factor $1<C_n<1+o(1)$ is known to often improve performances in practice when the sample size is small (see Section~6.3.2 of \citealp{Arl:2009:RP}), and
(ii) minimal penalty algorithms like Algorithm~\ref{algo1} are conjectured to overpenalize slightly when $n$ is small or the noise-level is large \citep{Ler:2010:iid} (as confirmed by Figure~\ref{fig_variation_sigma_2}).
Interestingly, this phenomenon is stronger for single-task estimators (differences are smaller in Figure~\ref{fig_variation_p_2}) and disappears when $n$ is large enough (Figure~\ref{fig_variation_n_3}), which is consistent with the heuristic motivating multi-task learning: ``increasing the number of tasks $p$ amounts to increase the sample size''.

Figures~\ref{fig_variation_n_2} and~\ref{fig_variation_n_3} show  that our procedure works well with small $n$, and that increasing $n$ does not seem to significantly improve the performance of our estimators, except in the single-task setting with $\S$ known, where the over-penalization phenomenon discussed above disappears. 

Table \ref{table_clustering} shows  that using the multitask procedure improves the estimation accuracy, both in the clustering setting and in the segmentation setting. The last line of Table \ref{table_clustering} does not show that the clustering setting improves over the ``segmentation into intervals'' one, which was awaited if  a model close to the oracle is selected in both cases.

Table \ref{table_cv} finally shows that our parameter tuning procedure outperforms 5-fold cross-validation.

\section{Conclusion and Future Work}
This paper shows that taking into account the unknown similarity between $p$ regression tasks can be done optimally (Theorem~\ref{thm_oracle_disc}).
The crucial point is to estimate the $p \times p$ covariance matrix $\S$ of the noise (covariance between tasks), in order to learn the task similarity matrix $M$.
Our main contributions are twofold.
First, an estimator of $\S$ is defined in Section~\ref{estimation_Sigma}, where non-asymptotic bounds on its error are provided under mild assumptions on the mean of the sample (Theorem~\ref{propmulti}).
Second, we show an oracle inequality (Theorem \ref{thm_oracle_disc}), more particularly with a simplified estimation of $\S$ and increased performances when the matrices of $\M$ are jointly diagonalizable (which often corresponds to cases where we have a prior knowledge of what the relations between the tasks would be). We do plan to expand our results to larger sets $\M$, which may require new concentration inequalities and new optimization algorithms.

Simulation experiments show that our algorithm works with reasonable sample sizes, and that our multi-task estimator often performs much better than its single-task counterpart.
Up to the best of our knowledge, a theoretical proof of this point remains an open problem that we intend to investigate in a future work.

\acks{This paper was supported by grants from the Agence Nationale de la Recherche ({\sc Detect} project, reference ANR-09-JCJC-0027-01) and from the European Research Council (SIERRA Project ERC-239993).}

\appendix

We give in Appendix the proofs of the different results stated in Sections~\ref{sec_anal_th}, \ref{estimation_Sigma} and~\ref{ineg_oracle}.
The proofs of our main results are contained in Sections~\ref{proof_propmulti} and~\ref{proof_thmoracle}.

\section{Proof of Proposition~\ref{prodscalG} \label{dem_prop_1}}

\begin{proof}
It is sufficient to show that $\scal{\cdot}{\cdot}_{\G}$ is positive-definite on $\G$. Take $g \in \G$ and $S = (S_{i,j})_{1\leq i\leq j\leq p}$ the symmetric postive-definite matrix of size $p$ verifying $S^2 = M$, and denote $T = S^{-1} = (T_{i,j})_{1\leq i,j\leq p}$. Let $f$ be the element of $\G$ defined by $\forall i \in \{1\dots p\},~ g(\cdot,i) = \sum_{k=1}^n T_{i,k}f(\cdot,k)$. We then have:
\begin{align*}
\scal{g}{g}_{\G} &= \sum_{i=1}^p \sum_{j=1}^p M_{i,j}\scal{g(\cdot,i)}{g(\cdot,j)}_{\F} \\
                 &= \sum_{i=1}^p \sum_{j=1}^p \sum_{k=1}^p \sum_{l=1}^p M_{i,j}T_{i,k}T_{j,l}\scal{f(\cdot,k)}{f(\cdot,l)}_{\F}\\
                 &= \sum_{j=1}^p \sum_{k=1}^p \sum_{l=1}^p T_{l,j}\scal{f(\cdot,k)}{f(\cdot,l)}_{\F} \sum_{i=1}^p M_{j,i}T_{i,k}\\
                 &= \sum_{j=1}^p \sum_{k=1}^p \sum_{l=1}^p T_{l,j}\scal{f(\cdot,k)}{f(\cdot,l)}_{\F} (M\cdot T)_{j,k}\\
                 &= \sum_{k=1}^p \sum_{l=1}^p T_{l,j}\scal{f(\cdot,k)}{f(\cdot,l)}_{\F} \sum_{j=1}^p T_{l,j}(M\cdot T)_{j,k}\\
                 &= \sum_{k=1}^p \sum_{l=1}^p \scal{f(\cdot,k)}{f(\cdot,l)}_{\F} (T\cdot M\cdot T)_{k,l}\\
                 &= \sum_{k=1}^p \|f(\cdot,k)\|_{\F}^2.
\end{align*}
This shows that $\scal{g}{g}_{\G} \geq 0$ and that $\scal{g}{g}_{\G} = 0 \Rightarrow f=0 \Rightarrow g=0$.
\end{proof}

\section{Proof of Corollary~\ref{cor_GRKHS} \label{dem_cor_RKHS}}

\begin{proof}
If $(x,j) \in  \X \times \{1,\dots,p\}$, the application $(f^1,\dots,f^p) \mapsto f^j(x)$ is clearly continuous. We now show that $(\G,\scal{\cdot}{\cdot}_{\G})$ is complete. If $(g_n)_{n\in \NN}$ is a  Cauchy sequence of $\G$ and if we define, as in Section~\ref{dem_prop_1}, the functions $f_n$   by $\forall n \in \NN,~ \forall i \in \{1\dots p\},~ g_n(\cdot,i) = \sum_{k=1}^p T_{i,k}f_n(\cdot,k)$. The same computations show that $(f_n(\cdot,i))_{n\in \NN}$ are Cauchy sequences of $\F$, and thus converge. So the sequence $(f_n)_{n\in \NN}$ converges in $\G$, and $(g_n)_{n\in \NN}$ does likewise.
 \end{proof}

\section{Proof of Proposition~\ref{prop_noy_multi}\label{dem_prop_noy_multi}}

\begin{proof}
We define
\begin{equation*}
  \tilde{\Phi}(x,j) = M^{-1} \cdot
\begin{pmatrix}
  \d_{1,j}\Phi(x) \\
  \vdots \\
  \d_{p,j}\Phi(x) \\
\end{pmatrix}\virg
\end{equation*}
with $\d_{i,j} = \mathbf{1}_{i=j}$ being the Kronecker symbol, that is, $\d_{i,j} = 1 $ if $i=j$ and $0$ otherwise. We now show that $\tilde{\Phi}$ is the feature function of the RKHS. For $g \in \G$ and $(x,l) \in \X \times \{1,\dots,p\}$, we have:

\begin{align*}
  \scal{g}{\tilde{\Phi}(x,l)}_{\G} & = \sum_{j=1}^p \sum_{i=1}^p M_{j,i} \scal{g(\cdot,j)}{\tilde{\Phi}(x,l)^i}_{\F} \\
  & = \sum_{j=1}^p \sum_{i=1}^p \sum_{m=1}^p M_{j,i}M^{-1}_{i,m} \d_{m,l} \scal{g(\cdot,j)}{\Phi(x)}_{\F} \\
  & = \sum_{j=1}^p  \sum_{m=1}^p (M \cdot M^{-1})_{j,m} \d_{m,l} g(x,j) \\
  & = \sum_{j=1}^p  \d_{j,l} g(x,j) = g(x,l) \pt
\end{align*}
Thus we can write:

\begin{align*}
  \tilde{k}((x,i),(y,j)) & = \scal{\tilde{\Phi}(x,i)}{\tilde{\Phi}(y,j)}_{\G} \\
  & = \sum_{h=1}^p \sum_{h'=1}^p M_{h,h'} \scal{M^{-1}_{h,i}\Phi(x)}{M^{-1}_{h',j}\Phi(y)}_{\F}\\
  & = \sum_{h=1}^p \sum_{h'=1}^p M_{h,h'}M^{-1}_{h,i}M^{-1}_{h',j}K(x,y)\\
  & = \sum_{h=1}^p M^{-1}_{h,i}(M\cdot M^{-1})_{h,j}K(x,y)\\
  & = \sum_{h=1}^p M^{-1}_{h,i}\d_{h,j}K(x,y) = M^{-1}_{i,j} K(x,y) \pt
\end{align*}
\end{proof}

\section{Computation of the Quadratic Risk in Example \ref{ex_calcul_oracle}\label{app_calcul_oracle}}
We consider here that $f^1=\dots =f^p$. We use the set $\Mmultiens$:
\begin{equation*}
\Mmultiens \egaldef \set{ \Mmulti \paren{\l,\mu} = (\l+p\mu) I_p - \frac{\mu}{p} \boldsymbol{1}\boldsymbol{1}\trsp \, / \, (\l,\mu) \in (0,+\infty)^2 }
\end{equation*}
Using the estimator $\fh_{M} = A_{M} y$ we can then compute the quadratic risk using the bias-variance decomposition given in Equation~\eqref{eqfond1}:
\begin{equation*}
  \esp{\norr{\fh_{M} -f}^2} =\|(A_{M}-I_{np})f\|_2^2 + \tr(A_{M}\trsp A_{M}\cdot(\S \otimes I_n))  \pt
\end{equation*}
Les us denote by $(e_1,\dots,e_p)$ the canonical basis of $\R^p$. The eigenspaces of $p^{-1}\boldsymbol{1}\boldsymbol{1}\trsp$ are:
\begin{itemize}
  \item $\linspan{e_1+\dots+e_p}$ corresponding to eigenvalue $p$,
  \item $\linspan{e_2-e_1,\dots,e_p-e_1}$  corresponding to eigenvalue $0$.
\end{itemize}
Thus, with $\mutilde = \l+p\mu$ we can diagonalize in an orthonormal basis any matrix $M_{\l,\mu}\in\M$ as $M=P\trsp D_{\l,\mutilde}P$, with $D = D_{\l,\mutilde} = \diag\{\l,\mutilde,\dots,\mutilde\}$. Les us also diagonalise in an orthonormal basis $K$: $K = Q\trsp \Delta Q$, $\Delta = \diag\{\mu_1,\dots,\mu_n\}$. Thus we can write (see Properties \ref{prop_mult_kron} and \ref{prop_trsp_kron} for basic properties of the Kronecker product):
\begin{equation*}
A_M = A_{M_{\l,\mu}} = (P\trsp \otimes Q\trsp) \left[(D^{-1}\otimes \Delta)\left((D^{-1}\otimes \Delta)+npI_{np} \right)^{-1} \right](P \otimes Q) \pt
\end{equation*}
We can then note that $(D^{-1}\otimes \Delta)\left((D^{-1}\otimes \Delta)+npI_{np} \right)^{-1}$ is a diagonal matrix, whose  diagonal entry of index $(j-1)n+i$ ($i\in\{1,\dots,n\}$, $j\in\{1,\dots,p\}$) is
\begin{equation*}
\left\{ \begin{array}{lr} &\frac{\mu_i}{\mu_i+np\l} \textrm{ if } j=1\virg\\
                          &\frac{\mu_i}{\mu_i+np\mutilde} \textrm{ if } j>1\pt \end{array}\right.
\end{equation*}
We can now compute both bias and variance.
\begin{description}
\item[Bias:] We can first remark that $(P\trsp \otimes Q\trsp) = (P \otimes Q)\trsp$ is an orthogonal matrix and that $P\times \boldsymbol{1} = (1,0,\dots,0)\trsp$. Thus, as in this setting $f^1=\dots =f^p$, we have $f = \boldsymbol{1}\otimes (f^1(X_1),\dots,f^1(X_n))\trsp$ and $(P\trsp \otimes Q\trsp)f = (1,0,\dots,0)\trsp \otimes Q(f^1(X_1),\dots,f^1(X_n))\trsp$. To keep notations simple we note $Q(f^1(X_1),\dots,f^1(X_n))\trsp \egaldef (g_1,\dots,g_n)\trsp$. Thus
\begin{align*}
  \|(A_{M}-I_{np})f\|_2^2 &= \|(P \otimes Q)\trsp \left[(D^{-1}\otimes K)\left((D^{-1}\otimes K)+npI_{np} \right)^{-1} -I_{np}\right](P \otimes Q)f\|_2^2 \\
                         &= \|\left[(D^{-1}\otimes \Delta)\left((D^{-1}\otimes \Delta)+npI_{np} \right)^{-1} -I_{np}\right] \\
                         & ~~~~~~~~~~~~~~\times(1,0,\dots,0)\trsp \otimes (g_1,\dots,g_n)\trsp\|_2^2 \pt
\end{align*}
As only the  first $n$ terms of $(P \otimes Q)f$ are non-zero we can finally write
\begin{equation*}
  \|(A_{M}-I_{np})f\|_2^2 = \sum_{i=1}^{n} \left(\frac{np\l}{\mu_i+np\l}\right)^2 g_i^2 \pt
\end{equation*}
\item[Variance:] First note that
\begin{equation*}
  (P \otimes Q)(\S \otimes I_n)(P \otimes Q)\trsp = (P \S P\trsp \otimes I_n) \pt
\end{equation*}
We can also note that $\tilde{\Sigma} \egaldef P \S P\trsp$ is a symmetric positive definite matrix, with positive diagonal coefficients. Thus we can finally write
\begin{align*}
\tr(A_{M}\trsp A_{M}\cdot(\S \otimes I_n)) &= \tr\bigg(P \otimes Q)\trsp\left[(D^{-1}\otimes \Delta)\left((D^{-1}\otimes \Delta)+npI_{np} \right)^{-1}\right]^2 \\
                                          & ~~~~~~~~~~~~~~ \times (P \otimes Q)(\S \otimes I_n)  \bigg) \\
                                          &= \tr\bigg(\left[(D^{-1}\otimes \Delta)\left((D^{-1}\otimes \Delta)+npI_{np} \right)^{-1}\right]^2 \\
                                          &  ~~~~~~~~~~~~~~ \times(P \otimes Q) (\S \otimes I_n)(P \otimes Q)\trsp \bigg)\\
                                          &= \sum_{i=1}^n \left[\left(\frac{\mu_i}{\mu_i+np\l}\right)^2\tilde{\Sigma}_{1,1}+\left(\frac{\mu_i}{\mu_i+np\mutilde}\right)^2\sum_{j=2}^p\tilde{\Sigma}_{j,j}\right] \pt
\end{align*}
\end{description}
As noted at the end of Example \ref{ex_calcul_oracle} this leads to an oracle which has all its $p$ functions equal.

\subsection{Proof of Equation~\eqref{pen_id_HM} in Section~\ref{result_diagonalizable}} Let $M\in \mathcal{S}_p^{++}(\R)$, $P \in \mathcal{O}_p(\R)$ such that $M = P\trsp \diag(d_1,\dots,d_p)P$ and $\tilde{\S} = P\Sigma P\trsp$. We recall that $A_{\l}=K(K+n\l I_n)^{-1}$. The computations detailed above also show that the ideal penalty introduced in \Eq{defpenidmulti} can be written as
\begin{equation*}
 \pen_{\id}(M) = \frac{2\tr \big( A_{M}\cdot(\S \otimes I_n) \big) }{np} = \frac{2}{np} \left(\sum_{j=1}^p \tr(A_{pd_j})\tilde{\S}_{j,j}  \right) \pt
\end{equation*}

\section{Proof of Theorem~\ref{propmulti} \label{proof_propmulti}}
Theorem~\ref{propmulti} is proved in this section, after stating some classical linear algebra results (Section~\ref{sec.thm.propmulti.tools}).

\subsection{Some Useful Tools} \label{sec.thm.propmulti.tools}
We now give two properties of the Kronecker product, and then introduce a useful norm on $\mathcal{S}_p(\R)$, upon which we give several properties. Those are the tools needed to prove Theorem~\ref{propmulti}.

\begin{property} \label{prop_mult_kron}
  The Kronecker product is bilinear, associative and for every matrices $A,B,C,D$ such that the dimensions fit, $(A\otimes B)(C \otimes D) = (AC)\otimes (BD)$.
\end{property}

\begin{property} \label{prop_trsp_kron}
  Let $A \in \mathcal{M} _{n}(\R)$, $B \in\mathcal{M} _{B}(\R)$,  $(A \otimes B)\trsp = (A\trsp \otimes B\trsp)$.
\end{property}

\begin{definition}
We now introduce the norm $\nor{\cdot}$ on $\mathcal{S}_p(\R)$, which is the modulus of the eigenvalue of largest magnitude, and can be defined by
\begin{equation*}
  \nor{S} \egaldef \sup_{z \in \R^p, \norr{z}=1} \absj{ z\trsp Sz } \pt
\end{equation*}
\end{definition}
\noindent This norm has several interesting properties, some of which we will use are stated below.

\begin{property}
The norm $\nor{\cdot}$ is a matricial norm: $\forall (A,B) \in \mathcal{S}_p(\R)^2,~ \nor{AB} \leq \nor{A}\nor{B} $.
\end{property}
 \noindent We will use the following result, which is a consequence of the preceding Property.

\begin{equation*}
  \forall S \in \mathcal{S}_p(\R), ~ \forall T \in \mathcal{S}_p^{++}(\R), ~ \nor{T^{-\frac{1}{2}}ST^{-\frac{1}{2}}} \leq \nor{S}\nor{T^{-1}} \pt \\
\end{equation*}
We also have:

\begin{proposition}
\begin{equation*}
  \forall \S \in \mathcal{S}_p(\R),~ \nor{\S \otimes I_{n}} = \nor{\S} \pt
\end{equation*}
\end{proposition}

\begin{proof}
We can diagonalize $\S$ in an orthonormal basis: $\exists U \in \mathcal{O}_n(\R),~ \exists D = \textrm{Diag}(\mu_1,\dots,\mu_p),~ \S = U\trsp DU$. We then have, using the properties of the Kronecker product:
\begin{align*}
  \S \otimes I_{n} &= (U\trsp \otimes I_{n})(D \otimes I_{n}) (U \otimes I_{n}) \\
    &=(U \otimes I_{n})\trsp( D \otimes I_{n}) (U \otimes I_{n}) \pt
\end{align*}
We just have to notice that $U \otimes I_{n} \in \mathcal{O}_{np}(\R)$ and that:
\begin{equation*}
  D \otimes I_{n} = \textrm{Diag}(\underbrace{\mu_1,\dots,\mu_1}_{n \textrm{ times}},\dots,\underbrace{\mu_p,\dots,\mu_p}_{n \textrm{ times}}) \pt
\end{equation*}
\end{proof}
\noindent This norm can also be written in other forms:

\begin{property}\label{normat2}
If $M \in \mathcal{M}_n(\R)$, the operator norm $\norm{M}_2 \egaldef \sup_{t \in \R^n\backslash \{0\}} \left\{ \frac{\|Mt\|_2}{\|t\|_2}\right\}$ is equal to the greatest singular value of $M$: $\sqrt{\r(M\trsp M)}$. Henceforth, if $S$ is symmetric, we have  $\nor{S} = \|S\|_2$
\end{property}

\subsection{The Proof}
We now give a proof of Theorem~\ref{propmulti}, using Lemmas~\ref{lemmec}, \ref{lemmeC1} and~\ref{lemmeC2}, which are stated and proved in Section~\ref{sec.pr.Thm2.lemmes}.
The outline of the proof is the following:
\begin{enumerate}
\item Apply Theorem~\ref{thmono} to problem \eqref{Pz} for every $z \in \mathcal{Z}$ in order to
\item control $\norm{s-\zeta}_{\infty}$ with a large probability, where $s,\zeta \in \R^{p(p+1)/2}$ are defined by
\begin{align*}
s &\egaldef (\S_{1,1},\ldots,\S_{p,p},\S_{1,1}+\S_{2,2}+2\S_{1,2},\ldots,\S_{i,i}+\S_{j,j}+2\S_{i,j},\ldots)  \\
\mbox{and} \quad \zeta &\egaldef (a(e_1),\ldots,a(e_p),a(e_1+e_2),\ldots,a(e_1+e_p),a(e_2+e_3),\ldots,a(e_{p-1}+e_p))
\pt \end{align*}
\item Deduce that $\Sh=J(\zeta)$ is close to $\S=J(s)$ by controlling the Lipschitz norm of $J$.
\end{enumerate}

\begin{proof}
{\em 1. Apply Theorem~\ref{thmono}:} \newblock
We start by noticing that Assumption~\eqref{Hdf} actually holds true with all $\l_{0,j}$ equal.
Indeed, let $(\l_{0,j})_{1 \leq j \leq p}$ be given by  Assumption~\eqref{Hdf} and define $\l_0 \egaldef \min_{j=1, \ldots, p} \l_{0,j}$. Then, $\l_0 \in (0,+\infty)$ and $\df(\l_0)$ since all $\l_{0,j}$ satisfy these two conditions.
For the last condition, remark that for every $j\in\sset{1,\ldots,p}$, $\l_0 \leq \l_{0,j}$ and $\l \mapsto \snorm{(A_{\l} - I) F_{e_j}}_2^2$ is a nonincreasing function (as noticed in \citealp{Arl_Bac:2009:minikernel_long} for instance), so that
\begin{equation} \label{eq.prThm1.1} \frac{1}{n} \norm{(A_{\l_0} - I_n) F_{e_j}}_2^2 \leq \frac{1}{n} \norm{(A_{\l_{0,j}} - I_n) F_{e_j}}_2^2 \leq \Sigma_{j,j} \sqrt{\frac{\ln(n)}{n}} \pt \end{equation}
In particular, \Eq{eq.hypThm1} holds with $d_n=1$ for problem \eqref{Pz} whatever $z \in \sset{e_1, \ldots, e_p}$.

Let us now consider the case $z=e_i+e_j$ with $i \neq j \in \sset{1, \ldots, p}$.
Using \Eq{eq.prThm1.1} and that $F_{e_i+e_j} = F_{e_i} + F_{e_j}$, we have
\begin{equation*}
\norm{ (B_{\l_0}-I_n)F_{e_i+e_j} }_2^2  \leq \norm{ (B_{\l_0}-I_n)F_{e_i} }_2^2 + \norm{(B_{\l_0}-I_n)F_{e_j}}_2^2
 +2\scal{(B_{\l_0}-I_n)F_{e_i}}{(B_{\l_0}-I_n)F_{e_j}} \pt
\end{equation*}
The last term is bounded as follows:
\begin{align*}
2\scal{(B_{\l_0}-I_n)F_{e_i}}{(B_{\l_0}-I_n)F_{e_j}} &\leq 2\|(B_{\l_0}-I_n)F_{e_i}\|\cdot\|(B_{\l_0}-I_n)F_{e_j}\|\\
 & \leq 2\sqrt{n\ln(n)}\sqrt{\S_{i,i}\S_{j,j}} \\
 & \leq \sqrt{n\ln(n)} (\S_{i,i}+\S_{j,j}) \\
 & \leq (1+c(\Sigma)) \sqrt{n\ln(n)} (\S_{i,i}+\S_{j,j} + 2\S_{i,j}) \\
 & =  (1+c(\Sigma)) \sqrt{n\ln(n)} \sigma^2_{e_i+e_j}
 \virg
\end{align*}
because Lemma~\ref{lemmec} shows
\begin{equation*}
2 (\S_{i,i}+\S_{j,j}) \leq (1+c(\Sigma)) (\S_{i,i}+\S_{j,j} + 2\S_{i,j})\pt
\end{equation*}
Therefore, \Eq{eq.hypThm1} holds with $d_n=1+c(\Sigma)$ for problem \eqref{Pz} whatever $z \in \mathcal{Z}$.

{\em 2. Control $\norm{s-\zeta}_{\infty}$:} \newblock
Let us define
\[ \eta_1 \egaldef \b (2+\d) (1+c(\Sigma)) \sqrt{\frac{\ln(n)}{n}} \pt \]
By Theorem~\ref{thmono}, for every $z \in \mathcal{Z}$, an event $\Omega_z$ of probability greater than $1- n^{-\d}$ exists on which, if $n \geq n_0(\d)$,
\begin{equation*}
  (1-\eta_1) \s_z^2 \leq a(z) \leq (1+\eta_1) \s_z^2 \pt
\end{equation*}
So, on $\Omega \egaldef \bigcap_{z \in Z} \Omega_z$,
\begin{equation} \label{eq.prThm1.2}
\norm{\zeta-s}_{\infty} \leq \eta_1 \norm{s}_{\infty}
 \virg
\end{equation}
and $\Proba(\Omega) \geq 1 -  p(p+1)/2\times n^{-\d}$ by the union bound.
Let
\[ \norm{ \S }_{\infty} \egaldef \sup_{i,j}\absj{ \S_{i,j} } \quad \mbox{and} \quad  C_1(p) \egaldef \sup_{\S \in \mathcal{S}_p(\R)} \set{ \frac{\norm{\S}_{\infty}} {\nor{\S}} } \pt \]
Since  $\norm{ s }_{\infty} \leq 4 \norm{ \S }_{\infty}$ and $C_1(p)=1$ by Lemma~\ref{lemmeC1}, \Eq{eq.prThm1.2} implies that on $\Omega$,
\begin{equation} \label{eq.prThm1.3}
\norm{\zeta-s}_{\infty} \leq 4 \eta_1 \norm{\S}_{\infty} \leq 4 \eta_1 \nor{\S}
 \pt
\end{equation}

{\em 3. Conclusion of the proof:} \newblock
Let
\begin{equation*}
 C_2(p) \egaldef \sup_{\zeta \in \R^{p(p+1)/2}} \set{ \frac{\nor{J(\zeta)}}{\norm{\zeta}_{\infty}} } \pt
\end{equation*}
By Lemma~\ref{lemmeC2}, $C_2(p) \leq \frac{3}{2} p$.
By \Eq{eq.prThm1.3}, on $\Omega$,
\begin{equation} \label{eq.prThm1.4}
  \nor{\widehat{\S} - \S} = \nor{J(\zeta)-J(s)} \leq C_2(p) \norm{ \zeta - s }_{\infty}
  \leq 4 \eta_1 C_2(p) \nor{\S} \pt
\end{equation}
Since
\begin{equation*}
  \nor{\S^{-\frac{1}{2}}\widehat{\S}\S^{-\frac{1}{2}} - I_p} = \nor{\S^{-\frac{1}{2}}(\S - \widehat{\S})\S^{-\frac{1}{2}}} \leq \nor{\S^{-1}}\nor{\S - \widehat{\S}}\virg
\end{equation*}
and $\nor{\S}\nor{\S^{-1}} = c(\Sigma)$, \Eq{eq.prThm1.4} implies that on $\Omega$,
\begin{equation*}
  \nor{\S^{-\frac{1}{2}}\widehat{\S}\S^{-\frac{1}{2}} - I_p} \leq 4\eta_1 C_2(p)\nor{\S}\nor{\S^{-1}} = 4\eta_1 C_2(p) c(\Sigma) \leq 6\eta_1 p c(\Sigma) \pt
\end{equation*}
To conclude, \Eq{Schap_conv} holds on $\Omega$ with
\begin{equation} \label{eq.prThm1.5}
 \eta 
= 6 p c(\Sigma) \b (2+\d) (1+c(\Sigma)) \sqrt{\frac{\ln(n)}{n}}   \leq L_1(2+\d) p \sqrt{\frac{\ln(n)}{n}} c(\S)^2
\end{equation}
for some numerical constant $L_1$.
\end{proof}

\begin{rem}
As stated in \citet{Arl_Bac:2009:minikernel_long},  we need $\sqrt{n_0(\d)/\ln(n_0(\d))} \geq 504$ and \linebreak[4] $\sqrt{n_0(\d)}/\ln(n_0(\d)) \geq 24(290+\d)$.
\end{rem}

\begin{rem}
To ensure that the estimated matrix $\widehat{\S}$ is positive-definite we need that $\eta < 1$, that is,
\begin{equation*}
  \sqrt{\frac{n}{\ln(n)}} > 6\b(2+\d)p c(\Sigma) \paren{ 1+c(\Sigma) }  \pt
\end{equation*}
\end{rem}

\subsection{Useful Lemmas} \label{sec.pr.Thm2.lemmes}

\begin{lemma} \label{lemmec}
Let $p \geq 1$, $\Sigma \in \mathcal{S}_p^{++}(\R)$ and $c(\Sigma)$ its condition number. Then,
\begin{equation} \label{eqc}
  \forall 1 \leq i < j \leq p \, , \quad \S_{i,j} \geq - \frac{c(\Sigma)-1}{c(\Sigma)+1} \frac{\S_{i,i} + \S_{j,j}}{2} \virg
\end{equation}
\end{lemma}
\begin{rem}
The proof of Lemma~\ref{lemmec} shows the constant $\frac{c(\Sigma)-1}{c(\Sigma)+1}$ cannot be improved without additional assumptions on $\S$.
\end{rem}
\begin{proof}
It suffices to show the result when $p=2$. Indeed, \eqref{eqc} only involves $2\times 2$ submatrices $\widetilde{\S}(i,j) \in \mathcal{S}_2^{++}(\R)$  for which
\[ 1 \leq c\sparen{\widetilde{\S}} \leq c\paren{\S} \quad \mbox{hence} \quad 0 \leq \frac{c\sparen{\widetilde{\S}}-1}{c\sparen{\widetilde{\S}}+1} \leq \frac{c(\Sigma)-1}{c(\Sigma)+1} \pt \]
So, some $\theta\in\R$ exists such that
$\S = \nor{\S} R_{\theta}\trsp D R_{\theta}$ where
\begin{equation*}
  R_{\theta} \egaldef \begin{pmatrix} \cos(\theta)&  \sin(\theta) \\ -\sin(\theta) &\cos(\theta) \end{pmatrix}
\qquad
D = \begin{pmatrix}1 & 0 \\ 0 & \l \end{pmatrix}
\quad \mbox{and} \quad
\l \egaldef \frac{1}{c(\Sigma)}
\pt
\end{equation*}
Therefore,
\begin{equation*}
  \S = \nor{\S} \begin{pmatrix} \cos^2(\theta) + \l\sin^2(\theta) & \frac{1-\l}{2}\sin(2\theta) \\ \frac{1-\l}{2}\sin(2\theta)& \l\cos^2(\theta) + \sin^2(\theta)  \end{pmatrix} \pt
\end{equation*}
So, \Eq{eqc} is equivalent to
\[ \frac{(1-\l) \sin(2\theta)}{2} \geq - \frac{1-\l}{1+\l} \frac{1+\l}{2}  \virg \]
which holds true for every $\theta\in\R$, with equality for $\theta \equiv \pi/2$ (mod. $\pi$).
\end{proof}

\begin{lemma}
\label{lemmeC1}
For every $p \geq 1$, $C_1(p) \egaldef \sup_{\S \in \mathcal{S}_p(\R)} \frac{\|\S\|_{\infty}}{\nor{\S}} = 1\pt$
\end{lemma}

\begin{proof}
With $\S = I_p$ we have  $\|\S\|_{\infty} = \nor{\S} = 1$, so $C_1(p) \geq 1$.\\
Let us introduce $(i,j)$ such that $|\S_{i,j}| = \|\S\|_{\infty}$. We then have, with $e_k$ being the $k^{\textrm{th}}$ vector of the canonical basis of $\R^p$,
\begin{equation*}
 | \S_{i,j}| = |e_i \trsp \S e_j|  \leq |e_i \trsp \S e_i|^{1/2} |e_j \trsp \S e_j|^{1/2} \leq (\norr{\S}^{1/2})^2 \pt
\end{equation*}
\end{proof}

\begin{lemma}
\label{lemmeC2}
For every $p \geq 1$, let $C_2(p) \egaldef \sup_{\zeta \in \R^{p(p+1)/2}} \frac{\nor{J(\zeta)}}{\|\zeta\|_{\infty}}$. Then,
\begin{equation*}
  \frac{p}{4} \leq C_2(p)\leq \frac{3}{2}p\pt
\end{equation*}
\end{lemma}

\begin{proof}
For the lower bound, we consider
\begin{equation*}
  \zeta_1 = (\underbrace{1,\dots,1}_{p\textrm{ times}},\underbrace{4,\dots,4}_{\frac{p(p-1)}{2}\textrm{ times}})
\, , \quad \mbox{then} \quad
  J(\zeta_1) = \begin{pmatrix} 1 & \ldots & 1 \\ \vdots & \ddots & \vdots \\ 1 & \ldots & 1 \end{pmatrix}
\end{equation*}
so that $\nor{J(\zeta)} = p$ and $\|\zeta\|_{\infty} = 4$.

For the upper bound, we have for every $\zeta \in \R^{p(p+1)/2}$ and $z \in \R^p$  such that $\norr{z} = 1$
\begin{equation*}
z\trsp J(\zeta) z = \absj{ \sum_{1 \leq i,j \leq p} z_i z_j J(\zeta)_{i,j} } \leq \sum_{1 \leq i,j \leq p} \absj{z_{i}} \absj{z_{j}} \absj{ J(\zeta) } \leq \norm{J(\zeta)}_{\infty} \norm{z}_1^2 \pt
\end{equation*}
By definition of $J$, $\snorm{J(\zeta)}_{\infty} \leq 3/2 \norm{\zeta}_{\infty}$. Remarking that $\snorm{z}_1^2 \leq p \norm{z}_2^2$ yields the result.
\end{proof}

\section{Proof of Theorem~\ref{thm_oracle_disc} \label{proof_thmoracle}}
The proof of Theorem~\ref{thm_oracle_disc} is similar to the proof of Theorem 3 in \citet{Arl_Bac:2009:minikernel_long}. We give it here for the sake of completeness. We also show how to adapt its proof to demonstrate Theorem~\ref{thm_oracle_HM}. The two main mathematical results used here are Theorem \ref{propmulti} and a gaussian concentration inequality from \citet{Arl_Bac:2009:minikernel_long}.
\subsection{Key Quantities and their Concentration Around their Means}
\begin{definition}
We introduce, for $S \in \mathcal{S}_p^{++}(\R)$,
\begin{equation}\label{defM_o}
  \Mh_o(S) \in \argmin{M \in \M} \set{ \norr{\widehat{F}_{M} - Y} + 2\tr\left(A_M\cdot(S\otimes I_n) \right) }
\end{equation}
\end{definition}

\begin{definition}
Let  $S \in \mathcal{S}_p(\R)$, we note $S_+$ the symmetric matrix where the eigenvalues of $S$ have been thresholded at $0$. That is, if $S = U\trsp D U$, with $U \in \mathcal{O}_p(\R)$ and $D = \diag(d_1,\dots,d_p)$, then
\[ S_+ \egaldef U\trsp \diag\paren{\max\set{d_1,0 },\dots, \max\set{d_n,0} } U \pt \]
\end{definition}

\begin{definition}
For every $M \in \M$, we define
\begin{align*}
  b(M) &= \|(A_{M}-I_{np})f\|_2^2 \virg\\
  v_1(M) &= \esp{\scal{\e}{A_{M}\e}} = \tr(A_{M}\cdot(\S \otimes I_n))\virg \\
  \d_1(M) &= \scal{\e}{A_{M}\e} - \esp{\scal{\e}{A_{M}\e}} = \scal{\e}{A_{M}\e} - \tr(A_{M}\cdot(\S \otimes I_n))\virg \\
  v_2(M) &= \esp{\|A_{M}\e\|_2^2} = \tr(A_{M}\trsp A_{M}\cdot(\S \otimes I_n)) \virg\\
  \d_2(M) &= \|A_{M}\e\|_2^2 - \esp{\|A_{M}\e\|_2^2}  = \|A_{M}\e\|_2^2 - \tr(A_{M}\trsp A_{M}\cdot(\S \otimes I_n))\virg \\
  \d_3(M) &= 2\scal{A_{M}\e}{(A_{M}-I_{np})f}\virg \\
  \d_4(M) &= 2\scal{\e}{ (I_{np}-A_{M})f }\virg\\
  \widehat{\Delta}(M) &= -2\d_1(M) + \d_4(M) \pt
\end{align*}
\end{definition}

\begin{definition}
Let $C_A,C_B,C_C,C_D,C_E,C_F$  be fixed nonnegative constants. For every $x \geq 0$ we define the event
\begin{equation*}
  \O_x = \O_x(\M,C_A,C_B,C_C,C_D,C_E,C_F)
\end{equation*}
on which, for every $M \in \M$ and $\th_1,\th_2,\th_3,\th_4 \in (0,1]$:
\begin{align}
  |\d_1(M)| &\leq \th_1\tr\left(A_M\trsp A_M \cdot (\S \otimes I_n)\right) + (C_A + C_B\th_1^{-1})x\nor{\S} \label{ineqd1}\\
  |\d_2(M)| &\leq \th_2\tr\left(A_M\trsp A_M \cdot (\S \otimes I_n)\right) + (C_C + C_D\th_2^{-1})x\nor{\S}\label{ineqd2}\\
  |\d_3(M)| &\leq \th_3 \norr{(I_{np}-A_M)f}^2 + C_E\th_3^{-1}x\nor{\S}\label{ineqd3} \\
  |\d_4(M)| &\leq \th_4 \norr{(I_{np}-A_M)f}^2 + C_F\th_4^{-1}x\nor{\S}\label{ineqd4}
\end{align}
\end{definition}

Of key interest is the concentration of the empirical processes $\d_i$, uniformly over $M \in \M$. The following Lemma introduces such a result, when $\M$ contains symmetric matrices parametrized with their eigenvalues (with fixed eigenvectors).

\begin{lemma} \label{lemma_concentration}
 Let
\begin{equation*}
  C_A = 2,~ C_B = 1,~C_C = 2,~ C_D = 1,~C_E =306.25,~ C_F =306.25 \pt
\end{equation*}
Suppose that \eqref{HM} holds. Then $\mathbb{P}(\O_x(\M,C_A,C_B,C_C,C_D,C_E,C_F)) \geq 1-pe^{1027+\ln(n)}e^{-x}$. Suppose that \eqref{Hdisc} holds. Then $\mathbb{P}(\O_x(\M,C_A,C_B,C_C,C_D,C_E,C_F)) \geq 1-6p\card(\M)e^{-x}$.
\begin{equation*}
   \pt
\end{equation*}
\end{lemma}

\begin{proof}
\begin{description}
\item[First common step.] Let $M \in \M$, $P_M \in \mathcal{O}_p(\R)$ such that $M = P_M\trsp D P_M$, with $D = \diag(d_1,\dots,d_p)$. We can write:
\begin{align*}
  A_M = A_{d_1,\dots,d_p} &= (P_M\otimes I_n)\trsp \left[ (D^{-1}\otimes K) \left(D^{-1}\otimes K + npI_{np} \right)^{-1} \right]   (P_M\otimes I_n) \\
                     &= Q\trsp \tilde{A}_{d_1,\dots,d_p} Q \virg
\end{align*}
with $Q= P_M\otimes I_n$ and $\tilde{A}_{d_1,\dots,d_p} = (D^{-1}\otimes K)(D^{-1}\otimes K + npI_{np})^{-1}$. Remark that $\tilde{A}_{d_1,\dots,d_p}$ is block-diagonal, with diagonal blocks being $B_{d_1},\dots,B_{d_p}$ using the notations of Section~\ref{simple_reg}.  With $\tilde{\e} = Q\e = (\tilde{\e_1}\trsp,\dots,\tilde{\e_p}\trsp)\trsp$ and $\tilde{f} = Qf = (\tilde{f_1}\trsp,\dots,\tilde{f_p}\trsp)\trsp$ we  can write
\begin{align*}
  |\d_1(M)| &= \scal{\tilde{\e}}{\tilde{A}_{d_1,\dots,d_p}\tilde{\e}} - \esp{ \scal{\tilde{\e}}{\tilde{A}_{d_1,\dots,d_p}\tilde{\e}} } \virg\\
  |\d_2(M)| &= \norr{ \tilde{A}_{d_1,\dots,d_p}\tilde{\e} }^2 - \esp{ \norr{\tilde{A}_{d_1,\dots,d_p}\tilde{\e}}^2 }\virg \\
  |\d_3(M)| &= 2\scal{\tilde{A}_{d_1,\dots,d_p}\tilde{\e} }{(\tilde{A}_{d_1,\dots,d_p}-I_{np})\tilde{f}}\virg\\
  |\d_4(M)| &= 2\scal{\tilde{\e} }{(I_{np}-\tilde{A}_{d_1,\dots,d_p})\tilde{f}} \pt
\end{align*}
We can see that the quantities $\d_i$ decouple, therefore
\begin{align*}
  |\d_1(M)| &= \sum_{i=1}^p \scal{\tilde{\e}_i}{A_{pd_i}\tilde{\e}_i} - \esp{ \scal{\tilde{\e}_i}{A_{pd_i}\tilde{\e}} } \virg \\
  |\d_2(M)| &= \sum_{i=1}^p \norr{A_{pd_i}\tilde{\e}_i }^2- \esp{\norr{A_{pd_i}\tilde{\e}_i }^2} \virg \\
  |\d_3(M)| &= \sum_{i=1}^p 2\scal{A_{pd_i}\tilde{\e}_i }{(A_{pd_i} - I_n)\tilde{f}_i} \virg \\
  |\d_4(M)| &= \sum_{i=1}^p 2\scal{\tilde{\e}_i }{(I_n - A_{pd_i})\tilde{f}_i} \pt
\end{align*}
\item[Supposing \eqref{HM}.] Assumption \eqref{HM} implies that the matrix $P$ used above is the same for all the matrices $M$ of $\M$. Using Lemma 9 of \citet{Arl_Bac:2009:minikernel_long}, where we have $p$ concentration results on the sets $\tilde{\O}_i$, each of probability at least $1-e^{1027+\ln(n)}e^{-x}$  we can state that, on the set $\bigcap_{i=1}^p \tilde{\O}_i$, we have uniformly on $\M$
\begin{align*}
  |\d_1(M)| &\leq \sum_{i=1}^p \th_{1} \var[\tilde{\e}_i] \tr(A_{pd_i}\trsp A_{pd_i}) + (C_A + C_B\th_{1}^{-1})x\var[\tilde{\e}_i] \virg \\
  |\d_2(M)| &\leq \sum_{i=1}^p \th_{2} \var[\tilde{\e}_i] \tr(A_{pd_i}\trsp A_{pd_i}) + (C_C + C_D\th_{2}^{-1})x\var[\tilde{\e}_i] \virg \\
  |\d_3(M)| &\leq \sum_{i=1}^p \th_3 \norr{(I_n-A_{pd_i})\tilde{f}_i}^2 + C_E\th_3^{-1} x \var[\tilde{\e}_i]\virg \\
  |\d_4(M)| &\leq \sum_{i=1}^p \th_4 \norr{(I_n-A_{pd_i})\tilde{f}_i}^2 + C_F\th_4^{-1} x \var[\tilde{\e}_i]\pt \\
\end{align*}
\item[Supposing \eqref{Hdisc}.] We can use Lemma 8 of \citet{Arl_Bac:2009:minikernel_long} where we have $p$ concentration results on the sets $\tilde{\O}_{j,M}$, each of probability at least $1-6e^{-x}$  we can state that, on the set $\bigcap_{j=1}^p\bigcap_{M\in\M} \tilde{\O}_i$, we have uniformly on $\M$ the same inequalities written above.
\item[Final common step.] To conclude, it suffices to see that for every $i \in \{1,\dots,p\}$, $\var[\tilde{\e}_i] \leq \nor{\S}$.
\end{description}
\end{proof}

\subsection{Intermediate Result}
We first prove a general oracle inequality, under the assumption that  the penalty we use (with an estimator of $\S$) does not underestimate the ideal penalty (involving $\S$) too much.
\begin{proposition}\label{prop_thm_oracle}
Let $C_A,C_B,C_C,C_D,C_E \geq 0$ be fixed constants, $\gamma >0$, $\th_{S} \in [0,1/4)$ and $K_S \geq 0 $. On $\O_{\gamma \ln(n)}(\M,C_A,C_B,C_C,C_D,C_E)$, for every $S \in \mathcal{S}_{p}^{++}(\R)$ such that
\begin{equation} \label{hypS}
\begin{split}
 \tr\left( A_{\Mh_o(S)}\cdot((S-\S)\otimes I_n)\right) ~~~~~~~~~~~~~~~~~~~~~~~~~~~~~~~~~~~~~~~~~~~~~~~~~~~~~~~~~~~~~~~~~
\\\geq  -\th_{S}\tr\left( A_{\Mh_o(S)}\cdot(\S\otimes I_n)\right) \inf_{M \in \M} \left\{ \frac{b(M)+v_2(M)+K_S \ln(n) \nor{\S}}{v_1(M)} \right\}
\end{split}
\end{equation}
and for every $\th \in (0,(1-4\th_S)/2)$, we have:

\begin{multline} \label{res_prop_demo_th_2}
  \frac{1}{np} \norr{\fh_{\Mh_o(S)}-f}^2 \leq \frac{1+2\th}{1-2\th-4\th_{S}} \inf_{M \in \M}\left\{ \frac{1}{np} \norr{\widehat{F}_M-F}^2+\frac{2\tr\left(A_M\cdot ((S-\S)_+\otimes I_n)\right)}{np} \right\}\\
  +\frac{1}{1-2\th-4\th_S}\left[ (2C_A+3C_C+6C_D+6C_E+\frac{2}{\th}(C_B+C_F))\gamma+ \frac{\th_SK_S}{4} \right] \frac{\ln(n)\nor{\S}}{np}
\end{multline}
\end{proposition}

\begin{proof}
The proof of Proposition~\ref{prop_thm_oracle} is very similar to the one of Proposition 5 in \citet{Arl_Bac:2009:minikernel_long}.
First, we have
\begin{align}
  \norr{\fh_{M} -f}^2 &= b(M) + v_2(M) + \d_2(M) + \d_3(M) \label{eqfond1} \virg \\
  \norr{\fh_{M} -y}^2 &= \|\fh_{M} -f\|_2^2 - 2v_1(M) - 2\d_1(M) + \d_4(M) + \|\e\|_2^2 \pt\label{eqfond2}
\end{align}
Combining \Eq{defM_o} and \eqref{eqfond2}, we get:
\begin{equation} \label{generalstartingpoint}
  \begin{split}
  \norr{ \fh_{\Mh_o(S)}-f}^2  +2\tr\left( A_{\Mh_o(S)}\cdot((S-\S)_+\otimes I_n)\right)+\widehat{\Delta} (\Mh_o(S)) \\ \leq \inf_{M \in \M}\left\{ \norr{ \fh_{M}-f}^2 +2\tr\left( A_{M}\cdot((S-\S)\otimes I_n)\right) +\widehat{\Delta} (M) \right\} \pt
  \end{split}
\end{equation}
On the event $\O_{\gamma \ln(n)}$, for every $\th \in (0,1]$ and $M \in \M$, using \Eq{ineqd1} and \eqref{ineqd4} with $\th = \th_1 = \th_4$,
\begin{equation}\label{borne_Delta}
  |\widehat{\Delta}(M)| \leq \th (b(M)+v_2(M)) + (C_A+\frac{1}{\th}(C_B+C_F))\gamma \ln(n) \nor{\S} \pt
\end{equation}
Using \Eq{ineqd2} and \eqref{ineqd3} with $\th_2 = \th_3 = 1/2$ we get that for every $M \in \M$
Equation\begin{equation*}
  \norr{\widehat{F}_M-F}^2 \geq \frac{1}{2}(b(M) + v_2(M)) - (C_C+2C_D+2C_E)\gamma \ln(n) \nor{\S}\virg
\end{equation*}
which is equivalent to
\begin{equation}\label{maj_b+v_2}
  b(M) + v_2(M) \leq 2\norr{\widehat{F}_M-F}^2 + 2(C_C+2C_D+2C_E)\gamma \ln(n) \nor{\S}\pt
\end{equation}
Combining \Eq{borne_Delta} and \eqref{maj_b+v_2}, we get
\begin{equation*}
  |\widehat{\Delta}(M)| \leq 2\th\norr{\widehat{F}_M-F}^2 + \left(C_A + (2C_C+4C_D+4C_E)\th + (C_B+C_F)\frac{1}{\th} \right)\gamma \ln(n) \nor{\S}\pt
\end{equation*}
With \Eq{generalstartingpoint}, and with $C_1 = C_A$, $C_2 =2C_C+4C_D+4C_E$ and $C_3 =C_B+C_F$   we get
\begin{equation} \label{eq_apres_gsp}
  \begin{split}
  (1-2\th)\norr{ \fh_{\Mh_o(S)}-f}^2+2\tr\left( A_{\Mh_o(S)}\cdot((S-\S)_+\otimes I_n)\right) \leq \\\inf_{M \in \M}\left\{ \norr{ \fh_{M}-f}^2 +2\tr\left( A_{M}\cdot((S-\S)\otimes I_n)\right)\right\} + \left(C_1 + C_2\th + \frac{C_3}{\th} \right)\gamma \ln(n) \nor{\S}\pt
  \end{split}
\end{equation}
Using \Eq{hypS} we can state that
\begin{equation*}
  \tr\left( A_{\Mh_o(S)}\cdot((S-\S)\otimes I_n)\right) \geq  \frac{b(\Mh_o(S))+v_2(\Mh_o(S))+K_S \ln(n) \nor{\S}}{v_1(\Mh_o(S))} \tr\left( A_{\Mh_o(S)}\cdot(\S\otimes I_n)\right)
\end{equation*}
so that
\begin{equation*}
  \tr\left( A_{\Mh_o(S)}\cdot((S-\S)\otimes I_n)\right) \geq -\th_S\left( (b(\Mh_o(S))+v_2(\Mh_o(S)) + K_S \ln(n)\nor{S}\right) \virg
\end{equation*}
which then leads to \Eq{res_prop_demo_th_2} using \Eq{maj_b+v_2} and \eqref{eq_apres_gsp}.
\end{proof}

\subsection{The Proof Itself}
We now show Theorem~\ref{thm_oracle_disc} as a consequence of Proposition~\ref{prop_thm_oracle}. It actually suffices to show that $\widehat{\S}$ does not underestimate $\S$ too much, and that the second term in the infimum of \Eq{res_prop_demo_th_2} is negligible in front of the quadratic error $(np)^{-1}\snorm{ \fh_{M}-f}^2$.

\begin{proof}
On the event $\O$ introduced in Theorem~\ref{propmulti}, \Eq{Schap_conv} holds. Let
\begin{equation*}
   \gamma = pc(\Sigma) \paren{ 1+c(\Sigma) } \pt
\end{equation*}
By Lemma~\ref{lemma_CS} below, we have:
\begin{equation*}
  \inf_{M\in \M} \left\{ \frac{b(M)+v_2(M)+K_S\ln(n) \nor{\S}}{v_1(M)}\right\} \geq 2\sqrt{ \frac{K_S\ln(n)\nor{\S}}{n\tr(\S)}}\pt
\end{equation*}
We supposed Assumption \eqref{Hdisc} holds. Using elementary algebra it is easy to show that, for every symmetric positive definite matrices $A$, $M$ and  $N$ of size $p$, $M \succeq N$ implies that $\tr(AM) \geq \tr(AN)$. In order to have $\Mh_o(\widehat{\S})$ satisfying \Eq{hypS}, Theorem \ref{propmulti} shows that it suffices to have, for every $\th_S >0$,
\begin{equation*}
  2\th_S \sqrt{ \frac{K_S \ln(n)\nor{\S}}{n\tr(\S)}} = 6\beta(2+\d) \gamma \sqrt{ \frac{\ln(n)}{n}}\virg
\end{equation*}
which leads to the choice
\begin{equation*}
  K_S = \left( \frac{3\beta(\a+\d) \gamma \tr(\S)}{\th_S \nor{\S}}\right)^2 \pt
\end{equation*}
We now take $\th_S = \th = (9\ln(n))^{-1}$. Let $\Omega$ be the set given by Theorem \ref{propmulti}. Using \Eq{res_prop_demo_th_2} and requiring that $\ln(n) \geq 6$  we get, on the set $\tilde{\O} = \O \cap \O_{(\a+\d)\ln(n)}(\M,C_A,C_B,C_C,C_D,C_E,C_F) $ of probability $1-(p(p+1)/2+6pC)n^{-\d}$, using that $\a \geq 2$:
\begin{multline*}
  \frac{1}{np}\norr{\fh_{\Mh}-f} \leq \left( 1+\frac{1}{\ln(n)} \right) \inf_{M \in \M}\left\{ \frac{1}{np}\norr{ \fh_{M}-f}^2 +\frac{2\tr\left( A_{M}\cdot((\widehat{\S}-\S)_+\otimes I_n)\right)}{np} \right\} \\+ \left( 1-\frac{2}{3\ln(n)} \right)^{-1}\left[ 2C_A+3C_C+6C_D+6C_E+\ln(n)\left(18C_B+18C_F+\frac{729\beta^2\gamma^2\tr(\S)^2}{4\nor{\S}^2} \right) \right] \\ \times(\a + \d)^2\frac{\ln(n)^2\nor{\S}}{np} \pt
\end{multline*}
Using \Eq{eq.prThm1.5} and defining
$$
\eta_2 \egaldef 12 \b (\a+\d) \gamma \sqrt{\frac{\ln(n)}{n}}\virg
$$
we get
\begin{equation}\label{res_thm_oracle_var}
  \begin{split}
  \frac{1}{np}\norr{\fh_{\Mh}-f} \leq \left( 1+\frac{1}{\ln(n)} \right) \inf_{M \in \M}\left\{\frac{1}{np} \norr{ \fh_{M}-f}^2 + \eta_2\frac{\tr( A_{M}\cdot(\S\otimes I_n))}{np} \right\} \\+\left( 1-\frac{2}{3\ln(n)} \right)^{-1}\left[ 2C_A+3C_C+6C_D+6C_E+\ln(n)\left(18C_B+18C_F+\frac{729\beta^2\gamma^2\tr(\S)^2}{4\nor{\S}^2} \right) \right] \\ \times(\a + \d)^2\frac{\ln(n)^2\nor{\S}}{np}\pt
  \end{split}
\end{equation}

Now, to get a classical oracle inequality, we have to show that $\eta_2 v_1(M) =\eta_2 \tr( A_{M}\cdot(\S\otimes I_n))$ is negligible in front of  $\snorm{ \fh_{M}-f}^2$. Lemma~\ref{lemma_CS} ensures that:
\begin{equation*}
  \forall M \in \M \, , \, \forall x \geq 0 \, , \quad  2\sqrt{\frac{x\nor{\S}}{n\tr(\S)}}v_1(M) \leq v_2(M) + x\nor{\S} \pt
\end{equation*}
With $0<C_n<1$, taking $x$ to be equal to $72 \b^2  \ln(n)\gamma^2 \tr(\S) /(C_n\nor{\S})$ leads to
\begin{equation}\label{borne_v1}
  \eta_2 v_1(M) \leq 2C_n v_2(M) + \frac{72 \b^2 \ln(n)\gamma^2\tr(\S)}{C_n} \pt
\end{equation}
Then, since $v_2(M) \leq v_2(M)+b(M)$ and using also \Eq{eqfond1}, we get
\begin{equation*}
  v_2(M) \leq \norr{\fh_M-f}^2 + |\d_2(m)|+|\d_3(M)| \pt
\end{equation*}
On $\tilde{\O}$ we have that for every $\th \in (0,1)$, using \Eq{ineqd2} and \eqref{ineqd3},
\begin{equation*}
|\d_2(M)| + |\d_3(M)| \leq 2\th\left(\norr{\fh_M-f}^2-|\d_2(M)| - |\d_3(M)| \right) + (C_C+(C_D+C_E)\th^{-1})(\a+\d)\ln(n)\nor{\S} \virg
\end{equation*}
which leads to
\begin{equation*}
|\d_2(M)| + |\d_3(M)| \leq \frac{2\th}{1+2\th}\norr{\fh_M-f}^2 + \frac{C_C+(C_D+C_E)\th^{-1}}{1+2\th} (\a+\d)\ln(n)\nor{\S} \pt
\end{equation*}
Now, combining this equation with \Eq{borne_v1}, we get
\begin{multline*}
 \eta_2 v_1(M) \leq \left(1+\frac{4C_n\th}{1+2\th}\right)\norr{\fh_M-f}^2 + 2C_n\frac{C_C+(C_D+C_E)\th^{-1}}{1+2\th} (\a+\d)\ln(n)\nor{\S}\\ + \frac{72 \b^2 \ln(n)\gamma^2\tr(\S)}{C_n} \pt
\end{multline*}
Taking $\th = 1/2$ then leads to
\begin{multline*}
 \eta_2 v_1(M) \leq \left(1+C_n\right)\norr{\fh_M-f}^2 + C_n(C_C+2(C_D+C_E)) (\a+\d)\ln(n)\nor{\S}\\ + \frac{72 \b^2 \ln(n)\gamma\tr(\S)}{C_n} \pt
\end{multline*}
We now take $C_n = 1/\ln(n)$.
We now replace the constants $C_A$, $C_B$, $C_C$, $C_D$, $C_E$, $C_F$ by their values in Lemma~\ref{lemma_concentration} and  we get, for some constant $L_2$,
\begin{eqnarray*}
\begin{split}
\left( 1-\frac{2}{3\ln(n)} \right)^{-1}\left[ 1851.5+\ln(n)\left(5530.5+\frac{729 \beta^2 \gamma^2}{4\nor{\S}^2} \right) +616.5\left(1+\frac{1}{\ln(n)}\right)\frac{1}{\ln(n)} \right] \\
+ \frac{72 \b^2 \ln(n)\gamma^2\tr(\S)}{C_n}
 \leq L_2\ln(n)\gamma^2\frac{\tr(\S)^2}{\nor{\S}^2}
\end{split}
\end{eqnarray*}
From this we can deduce \Eq{resultatoracle_ht_proba} by noting that $\gamma \leq 2 p c(\S)^2$.

Finally we deduce an oracle inequality in expectation by noting that if $n^{-1}\snorm{f_{\Mh}-f}^2 \leq R_{n,\d}$ on $\tilde{\O}$, using Cauchy-Schwarz inequality
\begin{align}\label{maj_fM-f}
  \esp{\frac{1}{np}\norr{\fh_{\Mh}-f}^2} &= \esp{\frac{\mathbf{1}_{\tilde{\O}}}{np}\norr{\fh_{\Mh}-f}^2} + \esp{\frac{\mathbf{1}_{\tilde{\O}^c}}{np}\norr{\fh_{\Mh}-f}^2}\nonumber\\
                                             &\leq \esp{R_{n,\d}} + \frac{1}{np}\sqrt{\frac{4p(p+1)+6pC}{n^{\d}}}\sqrt{\esp{\norr{\fh_{\Mh}-f}^4}}\pt
\end{align}
We can remark that, since $\nor{A_{M}} \leq 1$,
\begin{equation*}
  \norr{\fh_M-f}^2 \leq 2 \norr{A_{M}\e}^2 + 2\norr{(I_{np}-A_M)f}^2 \leq 2 \norr{\e}^2 + 8\norr{f}^2 \pt
\end{equation*}
So
\begin{equation*}
\esp{\norr{\fh_{\Mh}-f}^4} \leq 12 \left(np\nor{\S} + 4\norr{f}^2 \right)^2 \virg
\end{equation*}
together with \Eq{res_thm_oracle_var} and \Eq{maj_fM-f}, induces \Eq{resultatoracle_esp}, using that for some constant $L_3>0$,
\begin{equation*} 
12\sqrt{\frac{p(p+1)/2+6pC}{n^{\d}}}\left(\nor{\S} + \frac{4}{np}\norr{f}^2 \right) \leq L_3\frac{\sqrt{p(p+C)}}{n^{\d/2}}\left( \nor{\S}+\frac{1}{np}\norr{f}^2\right)\pt
\end{equation*}
\end{proof}

\begin{lemma}\label{lemma_CS}
Let $n,p\geq 1$ be two integers, $x \geq 0$ and $\S \in \mathcal{S}_p^{++}(\R)$.
Then,
\begin{equation*}
  \inf_{A \in \M_{np}(\R), \nor{A} \leq 1}\left\{ \frac{\tr (A \trsp A\cdot(\S \otimes I_n)) + x \nor{\S} }{\tr (A\cdot (\S \otimes I_n)) } \right\} \geq  2\sqrt{ \frac{x\nor{\S}}{n\tr(\S)}}
\end{equation*}
\end{lemma}
\begin{proof}
First note that the bilinear form on $\M_{np}(\R)$,  $(A,B) \mapsto \tr (A\trsp B \cdot (\S \otimes I_n))$ is a scalar product.  By Cauchy-Schwarz inequality, for every $A \in \M_{np}(\R)$,
\begin{equation*}
  \tr(A\cdot (\S \otimes I_n))^2 \leq \tr(\S \otimes I_n) \tr (A \trsp A\cdot(\S \otimes I_n)) \pt
\end{equation*}
Thus, since $\tr(\S \otimes I_n) = n\tr(\S)$, if $c = \tr(A\cdot (\S \otimes I_n))>0$,
\begin{equation*}
  \tr (A \trsp A\cdot(\S \otimes I_n)) \geq \frac{c^2}{n\tr(\S)} \pt
\end{equation*}
Therefore
\begin{align*}
  \inf_{A \in \M_{np}(\R), \nor{A} \leq 1}\left\{ \frac{\tr (A \trsp A\cdot(\S \otimes I_n)) + x \nor{\S} }{\tr (A\cdot (\S \otimes I_n)) } \right\} & \geq  \inf_{c >0}\left\{\frac{c}{n\tr(\S)} + \frac{ x \nor{\S}}{c} \right\}\\
  &\geq 2\sqrt{ \frac{x\nor{\S}}{n\tr(\S)}} \pt
\end{align*}
\end{proof}

\subsection{Proof of Theorem \ref{thm_oracle_HM}}

We now prove Theorem \ref{thm_oracle_HM}, first by proving that $\hat{\S}_{\textrm{HM}}$ leads to a sharp enough approximation of the penalty.

\begin{lemma}\label{lemma_Sh_HM}
Let  $\Sh_{\textrm{HM}}$ be defined as in Definition \ref{def_Sh_HM}, $\a = 2$, $\kappa>0$ be the numerical constant defined in Theorem~\ref{thmono} and assume \eqref{Hdf} and \eqref{HM} hold.
  For every $\d \geq 2$, a constant $n_0(\d)$, an absolute constant $L_1>0$ and an event $\Omega$ exist such that $\Proba(\Omega_{\textrm{HM}}) \geq 1-  p n^{-\d}$ and for every $n \geq n_0(\d)$, on $\Omega_{\textrm{HM}}$, for every $M$ in $\M$
  \begin{align} \label{Schap_conv_HM}
    & \hspace{-0.95cm}  (1-\eta)\tr(A_M\cdot(\Sigma\otimes I_n)) \leq \tr(A_M\cdot(\Sh_{\textrm{HM}}\otimes I_n))  \leq (1+\eta)\tr(A_M\cdot(\Sigma\otimes I_n)) \virg \\
    \mbox{where} \qquad  \eta &\egaldef L_1(\a+\d)  \sqrt{\frac{\ln(n)}{n}} \notag
    \pt
  \end{align}
\end{lemma}

\begin{proof}
  Let $P$ be defined by \eqref{HM}. Let $M \in \M$, and $(d_1,\dots,d_p)\in (0,+\infty)^p$ such that \linebreak[4] $M = P\trsp\diag(d_1,\dots,d_p)P$. Thus, as shown in Section \ref{app_calcul_oracle}, we have with $\tilde{\S} = P\S P\trsp$:
$$
  \tr(A_M\cdot(\Sigma\otimes I_n)) = \sum_{j=1}^p \tr(A_{pd_j}) \tilde{\Sigma}_{j,j} \pt
$$
let $\tilde{\s}_j$ be defined as in Definition \ref{def_Sh_HM} (and thus $\Sh_{\textrm{HM}} = P\diag(\tilde{\s}_1,\dots,\tilde{\s}_p)P\trsp$), we then have by Theorem \ref{thmono} that for every $j \in \{1,\dots,p\}$ an event $\Omega^j$ of probability $1-\kappa n^{-\delta}$ exists such that on $\Omega^j$ $|\tilde{\Sigma}_{j,j}-\tilde{\s}_j| \leq \eta \tilde{\Sigma}_{j,j}$. Since
\begin{equation*}
  \tr(A_M\cdot(\Sh_{\textrm{HM}}\otimes I_n)) = \sum_{j=1}^p \tr(A_{pd_j}) \tilde{\s}_{j} \virg
\end{equation*}
taking $\Omega_{\textrm{HM}} = \cap_{j=1}^p \Omega^j$ suffices to conclude.
\end{proof}
\begin{proof}[{\bf of Theorem \ref{thm_oracle_disc}}]
 Adapting the proof of Theorem \ref{thm_oracle_disc} to Assumption \eqref{HM} first requires to take $\gamma = 1$ as Lemma \ref{lemma_Sh_HM} allows us. It then suffices to take the set \linebreak[4] $\tilde{\O} = \O_{\textrm{HM}} \cap \O_{(2+\d)\ln(n)}(\M,C_A,C_B,C_C,C_D,C_E,C_F) $ (thus replacing $\a$ by $2$) of probability $1-(p(p+1)/2+p)n^{-\d} \geq 1- p^2n^{-\d}$---supposing $p \geq 2$---if we  require that $2\ln(n) \geq 1027$.

To get to the oracle inequality in expectation we use the same technique than above, but we note that $\sqrt{\mathbb{P}(\tilde{\O}^c)} \leq \tilde{L_4} \times p/n^{\d / 2}$. We can finally define the constant $L_4$ by:
\begin{equation*}
L_3\tr(\S) (2 + \d)^2\frac{p\ln(n)^3}{np} +\frac{p}{n^{\d/2}}\nor{\S} \leq L_4 \gamma^2\tr(\S) (\a + \d)^2\frac{p\ln(n)^3}{np} \pt
\end{equation*}
\end{proof}
\bibliography{biblio_en.bib}

\begin{thebibliography}{26}
\providecommand{\natexlab}[1]{#1}
\providecommand{\url}[1]{\texttt{#1}}
\expandafter\ifx\csname urlstyle\endcsname\relax
  \providecommand{\doi}[1]{doi: #1}\else
  \providecommand{\doi}{doi: \begingroup \urlstyle{rm}\Url}\fi

\bibitem[Akaike(1970)]{Aka70}
Hirotogu Akaike.
\newblock Statistical predictor identification.
\newblock \emph{Annals of the Institute of Statistical Mathematics},
  22:\penalty0 203--217, 1970.

\bibitem[Ando and Zhang(2005)]{Ando:2005:FLP:1046920.1194905}
Rie~Kubota Ando and Tong Zhang.
\newblock A framework for learning predictive structures from multiple tasks
  and unlabeled data.
\newblock \emph{Journal of Machine Learning Research}, 6:\penalty0 1817--1853,
  December 2005.
\newblock ISSN 1532-4435.

\bibitem[Argyriou et~al.(2008)Argyriou, Evgeniou, and
  Pontil]{DBLP:journals/ml/ArgyriouEP08}
Andreas Argyriou, Theodoros Evgeniou, and Massimiliano Pontil.
\newblock Convex multi-task feature learning.
\newblock \emph{Machine Learning}, 73\penalty0 (3):\penalty0 243--272, 2008.

\bibitem[Arlot(2009)]{Arl:2009:RP}
Sylvain Arlot.
\newblock Model selection by resampling penalization.
\newblock \emph{Electron. J. Stat.}, 3:\penalty0 557--624 (electronic), 2009.
\newblock ISSN 1935-7524.
\newblock \doi{10.1214/08-EJS196}.

\bibitem[Arlot and Bach(2011)]{Arl_Bac:2009:minikernel_long}
Sylvain Arlot and Francis Bach.
\newblock Data-driven calibration of linear estimators with minimal penalties,
  July 2011.
\newblock arXiv:0909.1884v2.

\bibitem[Arlot and Massart(2009)]{Arl_Mas:2009:pente}
Sylvain Arlot and Pascal Massart.
\newblock Data-driven calibration of penalties for least-squares regression.
\newblock \emph{Journal of Machine Learning Research}, 10:\penalty0 245--279
  (electronic), 2009.

\bibitem[Aronszajn(1950)]{Aronszajn}
Nachman Aronszajn.
\newblock Theory of reproducing kernels.
\newblock \emph{Transactions of the American Mathematical Society}, 68\penalty0
  (3):\penalty0 337--404, May 1950.

\bibitem[Bakker and Heskes(2003)]{Bakker:2003:TCG:945365.945370}
Bart Bakker and Tom Heskes.
\newblock Task clustering and gating for bayesian multitask learning.
\newblock \emph{Journal of Machine Learning Research}, 4:\penalty0 83--99,
  December 2003.
\newblock ISSN 1532-4435.
\newblock \doi{http://dx.doi.org/10.1162/153244304322765658}.

\bibitem[Birg{\'e} and Massart(2007)]{BirgeMassart_min_pen}
Lucien Birg{\'e} and Pascal Massart.
\newblock Minimal penalties for {G}aussian model selection.
\newblock \emph{Probability Theory and Related Fields}, 138:\penalty0 33--73,
  2007.

\bibitem[Brown and Zidek(1980)]{Brown_Zidek_1980}
Philip~J. Brown and James~V. Zidek.
\newblock Adaptive multivariate ridge regression.
\newblock \emph{The Annals of Statistics}, 8\penalty0 (1):\penalty0 pp. 64--74,
  1980.
\newblock ISSN 00905364.

\bibitem[Caruana(1997)]{Caruana:1997:ML:262868.262872}
Rich Caruana.
\newblock Multitask learning.
\newblock \emph{Machine Learning}, 28:\penalty0 41--75, July 1997.
\newblock ISSN 0885-6125.
\newblock \doi{10.1023/A:1007379606734}.

\bibitem[Evgeniou et~al.(2005)Evgeniou, Micchelli, and Pontil]{Evgeniou}
Theodoros Evgeniou, Charles~A. Micchelli, and Massimiliano Pontil.
\newblock Learning multiple tasks with kernel methods.
\newblock \emph{Journal of Machine Learning Research}, 6:\penalty0 615--637,
  2005.

\bibitem[Gasso et~al.(2009)Gasso, Rakotomamonjy, and Canu]{gasso2009}
Gilles Gasso, Alain Rakotomamonjy, and St\'ephane Canu.
\newblock Recovering sparse signals with non-convex penalties and dc
  programming.
\newblock \emph{IEEE Trans. Signal Processing}, 57\penalty0 (12):\penalty0
  4686--4698, 2009.

\bibitem[Horn and Johnson(1991)]{Horn_Johnson_Matrix_analysis}
Roger~A. Horn and Charles~R. Johnson.
\newblock \emph{Topics in Matrix Analysis}.
\newblock Cambridge University Press, 1991.
\newblock ISBN 9780521467131.

\bibitem[Jacob et~al.(2008)Jacob, Bach, and Vert]{jacob:clustered}
Laurent Jacob, Francis Bach, and Jean-Philippe Vert.
\newblock Clustered multi-task learning: A convex formulation.
\newblock \emph{Computing Research Repository}, pages --1--1, 2008.

\bibitem[Lerasle(2011)]{Ler:2010:iid}
Matthieu Lerasle.
\newblock Optimal model selection in density estimation.
\newblock \emph{Ann. Inst. H. Poincar{\'e} Probab. Statist.}, 2011.
\newblock ISSN 0246-0203.
\newblock Accepted. arXiv:0910.1654.

\bibitem[Liang et~al.(2010)Liang, Bach, Bouchard, and
  Jordan]{liang10regularization}
Percy Liang, Francis Bach, Guillaume Bouchard, and Michael~I. Jordan.
\newblock Asymptotically optimal regularization in smooth parametric models.
\newblock In \emph{Advances in Neural Information Processing Systems}, 2010.

\bibitem[Lounici et~al.(2010)Lounici, Pontil, Tsybakov, and van~de
  Geer]{Lounici:1277345}
Karim Lounici, Massimiliano Pontil, Alexandre~B. Tsybakov, and Sara van~de
  Geer.
\newblock Oracle inequalities and optimal inference under group sparsity.
\newblock Technical Report arXiv:1007.1771, Jul 2010.
\newblock Comments: 37 pages.

\bibitem[Lounici et~al.(2011)Lounici, Pontil, van~de Geer, and
  Tsybakov]{lounici2011oracle}
Karim Lounici, Massimiliano Pontil, Sarah van~de Geer, and Alexandre Tsybakov.
\newblock Oracle inequalities and optimal inference under group sparsity.
\newblock \emph{The Annals of Statistics}, 39\penalty0 (4):\penalty0
  2164--2204, 2011.

\bibitem[Mallows(1973)]{mallows1973}
Colin~L. Mallows.
\newblock Some comments on {$\textrm{C}_{\textrm{P}}$}.
\newblock \emph{Technometrics}, pages 661--675, 1973.

\bibitem[Obozinski et~al.(2011)Obozinski, Wainwright, and
  Jordan]{Obozinski_Wainwright_Jordan_2011}
Guillaume Obozinski, Martin~J. Wainwright, and Michael~I. Jordan.
\newblock Support union recovery in high-dimensional multivariate regression.
\newblock \emph{The Annals of Statistics}, 39\penalty0 (1):\penalty0 1--17,
  2011.

\bibitem[Rasmussen and Williams(2006)]{GP}
Carl~E. Rasmussen and Christopher~K.I. Williams.
\newblock \emph{Gaussian Processes for Machine Learning}.
\newblock MIT Press, 2006.

\bibitem[Sch\"olkopf and Smola(2002)]{Schölkopf_Smola_learning}
Bernhard Sch\"olkopf and Alexander~J. Smola.
\newblock \emph{Learning with Kernels: Support Vector Machines, Regularization,
  Optimization, and Beyond}.
\newblock Adaptive Computation and Machine Learning. MIT Press, Cambridge, MA,
  USA, 12 2002.

\bibitem[Thrun and O'Sullivan(1996)]{Thrun96g}
Sebastian Thrun and Joseph O'Sullivan.
\newblock Discovering structure in multiple learning tasks: The {TC} algorithm.
\newblock \emph{Proceedings of the 13th International Conference on Machine
  Learning}, 1996.

\bibitem[Wahba(1990)]{Wah:1990}
Grace Wahba.
\newblock \emph{Spline Models for Observational Data}, volume~59 of
  \emph{CBMS-NSF Regional Conference Series in Applied Mathematics}.
\newblock Society for Industrial and Applied Mathematics (SIAM), Philadelphia,
  PA, 1990.
\newblock ISBN 0-89871-244-0.

\bibitem[Zhang(2005)]{Zhang}
Tong Zhang.
\newblock Learning bounds for kernel regression using effective data
  dimensionality.
\newblock \emph{Neural Computation}, 17(9):\penalty0 2077--2098, 2005.

\end{thebibliography}
\end{document}